\documentclass[11pt]{amsart}
\usepackage{amsmath,amssymb,amsfonts,url,mathptmx,esint}
\usepackage[notref,notcite]{showkeys}
\numberwithin{equation}{section}

\newtheorem{thm}{Theorem}[section]
\newtheorem{lem}{Lemma}[section]
\newtheorem{rem}{Remark}[section]
\newtheorem{prop}{Proposition}[section]

\begin{document}
\title[Chern-Simons]{The domain geometry and the bubbling phenomenon of rank two Gauge theory} \subjclass{ }
\keywords{Green's function, Chern-Simons, Gauge theory, fully bubbling solutions}

\author{Hsin-Yuan Huang}
\address{Department of Mathematics\\
        National Sun Yat-sen Universtiy\\
        Kaoshiung, Taiwan\\
\&     National Center for Theoretical Sciences, Mathematics Division\\
Rm 203, Astro-Math. Bldg., National Taiwan University \\
Taipei, Taiwan}\email{hyhuang@math.nsysu.edu.tw}

\author{Lei Zhang}
\address{Department of Mathematics\\
        University of Florida\\
        358 Little Hall P.O.Box 118105\\
        Gainesville FL 32611-8105}
\email{leizhang@ufl.edu}

\date{\today}

\begin{abstract} Let $\Omega$ be a flat torus and $G$ be the green's function of $-\Delta$ on $\Omega$. One intriguing mystery of $G$ is how the number of its critical points is related to blowup solutions of certain PDEs. In this article we prove that for the following equation that describes a Chern-Simons model in Gauge theory:
\begin{equation}\label{e103}
\left\{
\begin{array}{ll}
\Delta u_1+\frac{1}{\varepsilon^2}e^{u_2}(1-e^{u_1})=8\pi\delta_{p_{1}}\\
\Delta u_2+\frac{1}{\varepsilon^2}e^{u_1}(1-e^{u_2})=8\pi\delta_{p_{2}}
\end{array}
\text{  in  }\quad \Omega\right., \quad p_1-p_2 \mbox{ is a half period},
\end{equation}
if fully bubbling solutions of Liouville type exist, $G$ has exactly three critical points. In addition we establish necessary conditions for the existence of fully bubbling solutions with multiple bubbles.
\end{abstract}


\maketitle

\section{Introduction}

Let $\Omega=\mathbb{R} \slash \mathbb{Z}\omega_1 \times \mathbb{R} \slash \mathbb{Z}\omega_2 $ be a flat torus, $\Delta$ be the Laplace-Beltrami operator with the property $-\Delta\ge 0$ and we use $G(\cdot,\cdot)$ to denote the Green's function of $-\Delta$
over $\Omega$  with periodical boundary condition:
\begin{equation}\begin{cases}
-\Delta G(x,p)=\delta_p-\frac{1}{|\Omega|},\\
G(x,p)\text{  is  doubly periodic on } \partial \Omega,
\end{cases}
\end{equation}
where $|\Omega|$ is the measure of $\Omega$, $\delta_p$ stands for the Dirac mass at $p\in \Omega$. Although $G(x,p)$ can be explicitly solved in terms of elliptic functions, it was only found out recently that $G(x,p)$ has either three or five critical points, as a function of $x$. In their celebrated work \cite{LW1} Lin and Wang reveal the surprising ties between the number of critical points of $G$ with the bubbling phenomena of the following mean field equation
\begin{equation}\label{e101}
\Delta u +e^{u}=\rho \delta_p.
\end{equation}
Note that $G(x,p)=G(x-p,0)$ because of the translation invariance. For simplicity, we may only consider $G(x)=G(x,0)$. Since $G(x)$ is even, it is easy to see that the half period $\frac{1}{2}\omega_1$, $\frac{1}{2}\omega_2$ and $\frac{1}{2}\omega_3=\frac{1}{2}(\omega_1+\omega_2)$ are critical points
and other critical points must appear in pairs with some symmetry. When $\rho=8\pi$,  Lin-Wang showed that there is only one extra pair of critical points which corresponds to an one-parameter scaling family of solutions to  \eqref{e101}.  We refer the readers to \cite{LW2,CKLW} for more recent developments on this direction. \par

 The first main result in this paper is to connect the number of critical points of the Green's function over  $\Omega$ with the bubbling phenomenon of the following system:
\begin{equation}\label{e102}
\begin{cases}
\Delta u_{1 }+\frac{1}{\varepsilon^2}e^{u_{2 }}(1-e^{u_{1 }})=8\pi \delta_{p_1}\\
\Delta u_{2 }+\frac{1}{\varepsilon^2}e^{u_{1 }}(1-e^{u_{2 }})=8\pi \delta_{p_2}
\end{cases}\text{   in   }\Omega.
\end{equation}
System (\ref{e102}) arises from  the relativistic self-dual  $[U(1)]^2$  Chern-Simons model proposed by Kim et al\cite{KLKLM}.
It is known that the number of critical points of the Green function depends on the geometry of the underlying flat torus. Our first main result is:

\begin{thm}\label{thm2}
Suppose  $(u_{1,\varepsilon},u_{2,\varepsilon})$ is a sequence of   fully bubbling solutions of  Liouville type to \eqref{e102}
and $p_1-p_2$ is  one of the half periods. Then the Green's function over  $\Omega$ has three critical points.
\end{thm}

The definition of fully bubbling solutions of the Liouville type will be given later in this section.

The proof of Theorem \ref{thm2} is based on recent developments on the Green's function over flat torus in \cite{LW1, LW2, CKLW} and necessary conditions of fully bubbling solutions in Theorem \ref{main1}.  Here we first give a brief introduction of the physical background of \eqref{e102}.

In the works of Hong, Kim and Pac\cite{HKP}, and Jackiw and Weinberg\cite{JW}, a model with one Chern-Simons gauge field was considered and selfdual Abelian Chern--Simons--Higgs vortices were constructed to describe anyonic solitons in 2+1 dimensions.
Later, Speilman et al.\cite{SFEG} observed no parity breaking in an experiment with high temperature superconductivity. Hagen \cite{Hagen} and  Wilczek \cite{Wil} indicated that the parity broken  may not  happen in a field theory with even number of Chern-Simons gauge fields.  One of the simplest models of this kind is the $[U(1)]^2$ Chern-Simons model of two Higgs fields, where each of them is coupled to one of two Chern--Simons fields.
%
The Lagrangian action  density for the  $[U(1)]^2$ Chern-Simons model is given by
\begin{equation}
 \mathcal{L}  =  -\frac{\varepsilon}{ 2}\epsilon^{\mu\nu\alpha}\left(A^{(1)}_\mu F^{(2)}_{\mu\nu}+A^{(2)}_\mu F^{(1)}_{\mu\nu}\right)+\sum\limits_{i=1}^2D_\mu\phi_i\overline{D^\mu\phi_i}-V(\phi_1, \phi_2), \label{a1}
\end{equation}
where $\varepsilon>0$ is a coupling parameter, $(A_\mu^{(1)})$ and $(A_\mu^{(2)})$   are two Abelian gauge fields with the electromagnetic fields $F^{(i)}_{\mu\nu}=\partial_\mu A_\nu^{(i)}-\partial_\nu A_\mu^{(i)}$,
$\phi_1$  and $\phi_2$ are two Higgs scalar fields with the  covariant derivatives $D_\mu\phi_i=\partial_\mu\phi_i-\mathrm{i} A^{(i)}_\mu\phi_i$($ \mu=0,1, 2,\,i=1,2 $), and the Higgs potential $ V(\phi_1,\phi_2) $ takes the following form:
 \begin{equation}
  V(\phi_1,\phi_2) =    \frac{1}{4\varepsilon^2}\left(|\phi_2|^2\left[|\phi_1|^2-1\right]^2+|\phi_1|^2\left[|\phi_2|^2-1\right]^2\right).
\end{equation}
 After a BPS reduction \cite{Bo,PS}, we obtain that the energy minimizer satisfies the following self-dual equation:
  \begin{equation}\label{er3}\left\{
\begin{array}{ll}
 &D_1\phi_k+\mathrm{i} D_2\phi_k=0, \quad k=1, 2 , \\
 & F^{(1)}_{12}  + \frac{1}{2\varepsilon^2}|\phi_2|^2\left(|\phi_1|^2-1\right)=0,\\
  & F^{(2)}_{12} + \frac{1}{2\varepsilon^2}|\phi_1|^2\left(|\phi_2|^2-1\right)=0.
   \end{array} \right.
\end{equation}
Let $u_{i,\varepsilon}=\ln |\phi_i|^2$, and $\{p_{i,1}, \dots, p_{i,{N_i}}\}$ be the zeros of $\phi_i$ for
 $i=1, 2$. Then $(u_{1,\varepsilon},u_{2,\varepsilon})$ satisfies
\begin{equation}\label{e001}
\left\{
\begin{array}{ll}
\Delta u_1+\frac{1}{\varepsilon^2}e^{u_2}(1-e^{u_1})=4\pi\sum_{i=1}^{N_1}\delta_{p_{1,i}}\\
\Delta u_2+\frac{1}{\varepsilon^2}e^{u_1}(1-e^{u_2})=4\pi\sum_{i=1}^{N_2}\delta_{p_{2,i}}
\end{array}
\text{  in  }\quad \Omega \right.,
\end{equation}
where $\Omega$ is either a flat torus or $\mathbb R^2$. See \cite{KLKLM,dz94,LPY} for the details of the derivation of \eqref{e001} from \eqref{er3}. In this paper, we consider the case of flat torus.  We refer the readers to \cite{LPY,LP,CCL,HL1,HHL1,CY2014} and reference therein for recent developments.\par

When $u_1\equiv u_2$ and $\{p_i^1\}_{i=1}^{N_1}=\{p_j^2\}_{j=1}^{N_2}$,  the system \eqref{e001}  reduces to the Abelian Chern--Simons equation with one Higgs particle proposed by  Kim--Pac\cite{HKP} and Jackiw--Weinberg\cite{JW},
\begin{equation} \label{eq17}
 \Delta u+\frac{1}{\varepsilon^2}e^u(1-e^u)=4\pi\sum_{i=1}^{N}\delta_{p_i},
\end{equation}
 which has been extensively studied for  more than  twenty years. We refer the readers to \cite{SY,T,CY1,C,Dunne1,T1,CFL,choe,CKL,DEFM} and the reference therein for more details.\par

Choe and Kim\cite{CK} established the Brezis-Merle type alternative \cite{BM} for \eqref{eq17}. They showed that \eqref{eq17} may have a sequence of solutions  satisfying the following:\\
\\{\it
 There is a finite set  $\{x_{1,\varepsilon},\cdots, x_{l,\varepsilon}\}$, $x_{j,\varepsilon}\in \Omega$, $j=1,\cdots,l$, such that, as $\varepsilon\to 0$
\begin{equation}\label{eq19}
u_{\varepsilon}(x_{j,\varepsilon})+2\ln \frac{1}{\varepsilon}\to\infty,\,\,j=1,\cdots, l,
\end{equation}
and
\begin{equation}\label{eq20}
u_{\varepsilon} +2\ln \frac{1}{\varepsilon}\to-\infty \text{ uniformly on any compact subset of }\Omega\setminus\{q_1,\cdots, q_l\},
\end{equation}
where $q_j=\lim_{\varepsilon}x_{j,\varepsilon}$.  }\par
Solutions of \eqref{eq17} satisfying \eqref{eq19} and \eqref{eq20} are called {\it bubbling solutions} or {\it blow-up solutions } and $q_j$ is called the blow-up point of the bubbling solution. We denote $\beta_{j,\varepsilon}=u_{\varepsilon}(x_{j,\varepsilon})$. It was shown in \cite{CK} that either $\beta_{j,\varepsilon}\to -\infty$ or  $\beta_{j,\varepsilon}$ is bounded. After   suitable re-scaling,  the bubbling solutions $u_{\varepsilon}$  converge  to an entire solution of either
\begin{equation}\label{eq411}
\Delta u+|x|^{2m}e^{u}(1-|x|^{m}e^u)=0,
\end{equation}
or
\begin{equation}\label{eq421}
\Delta u+|x|^{2m}e^{u}=0.
\end{equation}
Here, $m=0$ if $q_j$ is not a vortex point and $m=\#\{p_i: p_i=q_j\}$ if $q_j$ is a vortex point.
The bubbling solutions are  called {\it Chern-Simons type} if the limiting equation is \eqref{eq411}  and  {\it mean field type} if the limiting equation is \eqref{eq421}.  The existence and non-existence of bubbling solutions of \eqref{eq17} have been studied in a series of work of Lin and Yan\cite{LY,LY2,LY3}.\par

Set
\begin{equation}
u_{0,1}(x)=-4\pi\sum_{i=1}^{N_1} G(x,p_{1,i})\quad \text{              and              }\quad
u_{0,2}(x)=-4\pi\sum_{i=1}^{N_2} G(x, p_{2,i}).
\end{equation}
By the transformation  $u_i\to u_i+u_{0,i}$($i=1,2$), \eqref{e001} is reduced to the following system.

\begin{equation}\label{e002}
\begin{cases}
\Delta u_1+\frac{1}{\varepsilon^2}e^{u_2+u_{2,0}}(1-e^{u_1+u_{0,1}})=\frac{4\pi N_1}{|\Omega|}\\
\Delta u_2+\frac{1}{\varepsilon^2}e^{u_1+u_{1,0}}(1-e^{u_2+u_{0,2}})=\frac{4\pi N_2}{|\Omega|}
\end{cases}   \mbox{   in   }\Omega,
\end{equation}
where $u_1$ and $u_2$ are doubly periodic on $\partial\Omega$.\par

The sequence of solutions $(u_{1,\varepsilon},u_{2,\varepsilon})$ of \eqref{e002} is said to be {\it fully bubbling}  if there exist
$\{x_{1,j,\varepsilon}\}_{j=1}^k$ and $\{x_{2,j,\varepsilon}\}_{j=1}^k$ such that
\begin{itemize}
\item[(1)] $q_j=\lim_{\varepsilon\to 0}x_{1,j,\varepsilon}=\lim_{\varepsilon\to 0}x_{2,j,\varepsilon}$, $j=1,\cdots,k$. $u_{i,\epsilon}(x_{i,j,\epsilon})=\max_{B_d(q_j)}u_{i,\epsilon}$ for $i=1,2$, where
$B_d(q_j)$ is the ball centered at $q_j$ with radius $d$. Here we assume that $d>0$ is small so that $B_d(q_j)\cap B_d(q_m)=\emptyset$ if $j\neq m$.
\item[(2)]  $u_{i,\varepsilon}(x_{i,j,\varepsilon})+2\ln \frac{1}{\varepsilon} \to +\infty $  as $\varepsilon\to 0$, $i=1,2$, $j=1,\cdots,k$.
\item[(3)]   $u_{i,\varepsilon}(x)+2\ln \frac{1}{\varepsilon} \to -\infty $  as $\varepsilon\to 0$ uniformly on any compact set of $\Omega\setminus\{ q_1,\cdots,q_k\}$.
\end{itemize}

In this paper, we will  investigate  the {\it fully  bubbling solutions }to \eqref{e001} satisfying the following assumptions.
For  $i=1,2$, $j=1,\cdots,k$
\begin{itemize}
\item[(A1)] $\beta_{j,\varepsilon}=\max\{ u_{1,\varepsilon}(x_{1,j,\varepsilon}),u_{2,\varepsilon}(x_{2,j,\varepsilon})  \}\to -\infty$ as  $\varepsilon\to 0$,
\item[(A2)]  $ |u_{1,j,\varepsilon}(x_{1,j,\varepsilon})-u_{2,\varepsilon}(x_{2,j,\varepsilon})|=O(1)$
\item[(A3)]   $|x_{i,j,\varepsilon}-q_j|<C\varepsilon e^{-\frac{1}{2}\beta_{j,\varepsilon}}$ for some constant $C>0$.
\end{itemize}
Solutions satisfying (1),(2), (3) and (A1),(A2), (A3) are called
 {\it fully bubbling solutions of Liouville type} since after suitable rescaling,
their limiting equations are a Liouville system (see \eqref{e092} below).  \par

As mentioned before, one essential ingredient in the proof of Theorem \ref{thm2} is a necessary condition for the fully bubbling solutions of Liouville type whose blow-up points are regular points.  Before we state the second main result in this paper, we introduce some notations.  Let ${\bf q}=(q_1,\cdots,q_k)$, $q_i\in\mathbb{R}^2$, $i=1,\cdots,k$,
\begin{equation}
 G_1^*({\bf q})=\sum_{i=1}^k u_{0,1}(q_i)+8\pi \sum_{ 1\leq i<j\leq k}G(q_i,q_j)
 \end{equation}
and \begin{equation}
 G_2^*({\bf q})=\sum_{i=1}^k u_{0,2}(q_i)+8\pi \sum_{ 1\leq i<j\leq k}G(q_i,q_j).
\end{equation}
Denote the function $f_{i,j}$($i=1,2,$, $j=1,\cdots,k$) as follows.
\begin{equation}
f_{i,j}(x)= 8\pi\left(\gamma(x,q_j)-\gamma(q_j,q_j))+\sum_{l\not = j}(G(x,q_l)-G(q_j,q_l)\right)+u_{0,i}(x)-u_{0,i}(q_j),
\end{equation}
 where $\gamma(x,q)$  is the regular part of $G(x,p)$.   We define the quantity $\mathcal{D}^{(2)}({\bf q })$ as follows
\begin{equation}\label{e052}
\begin{split}
\mathcal{D}^{(2)}({\bf q})&=\lim_{\delta\to 0} \left( \sum_{j=1}^k  \frac{\rho_{1,j}}{ e^{u_{0,1}(q_1)}}
 \left(  \int_{\Omega_j\setminus B_{\delta}(q_j)   }\frac{e^{f_{1,j}}-1}{|x-q_{j }|^{4}}-  \int_{\mathbb{R}^2\setminus \Omega_j  } \frac{1}{|x-q_{j }|^{4}}  \right)\right.\\
 &\qquad +  \left.  \sum_{j=1}^k \frac{\rho_{2,j}}{e^{u_{0,2}(q_1)}}\left( \int_{\Omega_j\setminus B_{\delta} (q_j)  } \frac{e^{f_{2,j}}-1}{|x-q_{j }|^{4}}-  \int_{\mathbb{R}^2\setminus \Omega_j  } \frac{1}{|x-q_{j}|^{4} }  \right)\right)
\end{split}
\end{equation}
where  $\{\Omega_j\}_{j=1,\cdots,k}$ is any open set with
\begin{equation}\label{e108}
\begin{cases}
 (1)&\Omega_i\cap \Omega_j=\emptyset,\quad i\not =j ,\quad   i,j\in\{1,\cdots,k\} ,\\
 (2)&\cup_{j=1}^k \overline{\Omega_j}=\overline{\Omega} ,\\
 (3)& B_{\delta}(q_{j})\subset\subset \Omega_j , \quad j=1,\cdots,k ,
\end{cases}
\end{equation}
and
$
 \rho_{1,j}=e^{8\pi(\gamma(q_j,q_j)+\sum_{l\not=j}G(q_j,q_l)  ) +u_{0,1}(q_j)      }    ,   \quad
  \rho_{2,j}=e^{8\pi(\gamma(q_j,q_j)+\sum_{l\not=j}G(q_j,q_l)  ) +u_{0,2}(q_j)      }   . $
Note that throughout this article we use $B_\delta(q)$ to represent the ball centered at $q$ with radius $\delta$. If $q$ is the origin we use $B_\delta$.

\begin{thm}\label{main1}
Suppose $(u_{1,\varepsilon},u_{2,\varepsilon})$ is a sequence of fully bubbling solution of Liouville type to \eqref{e002} whose blow up set is $ {\bf q}=\{q_1,\cdots,  q_k\} \not\in\{p_{1,i},p_{2,i}\}_{i=1,2}$. Suppose in addition that $N_1=N_2=2k$. Then
\begin{enumerate}
\item $(u_{0,1}-u_{0,2})(q_i)=(u_{0,1}-u_{0,2})(q_j),\quad  1\leq i,j\leq k;$
\item
${\bf q}$ is  a critical point of $G_1^*$ and $G_2^*$;
\item
$\mathcal{D}^{(2)} ({\bf q})\leq 0$.
\end{enumerate}
\end{thm}

The proof of Theorem \ref{main1} is established by delicate blowup analysis, which is involved with handling difficulties that had never appeared before. One crucial estimate is to prove that the blowup solutions are very close to a sequence of global solutions. In order to achieve this goal we need to prove that the energy of the blowup solutions is very close to that of the global solution
{\bf on each component!} It is well known that for such systems there is only one Pohozaev identity, which is not enough for us to determine the energy of each component precisely. In this article we take advantage of the special form of the main equation to overcome this major difficulty. Our approach is likely to impact the study of bubbling solutions of many Liouville systems. The detailed discuss will be carried out in later sections. The following theorem of the first author actually proves that fully bubbling solutions of Liouville type do exist:

{\bf Theorem A} \cite{H-2015} \,\,{\it Let $\{q_1,\cdots, q_k\}$ be regular points.  Assume that
\begin{itemize}
\item[$(1)$] ${\bf q}$ be a critical point of $G_1^*$ and $G_2^*$
\item[$(2)$] ${\bf q}$ be non-degenerate critical point of $G_1^*+G_2^*$
\item[$(3)$] $\mathcal{D}^{(2)}({\bf q})<0$
\item[$(4)$] $(u_{0,1}-u_{0,2})(q_i)=(u_{0,1}-u_{0,2})(q_j)$, $1\leq i,j\leq k.$
\end{itemize}
Then there exists a sequence of fully bubbling solutions of Liouville type to \eqref{e002}.}\\
\\

The organization of this paper is as follows. The second section is devoted to the proof of Theorem \ref{thm2}.  In Section 3, we obtain the rough profile of the bubbling solutions.
In Section 4, we define the approximating solutions to the bubbling solution and establish estimates of error terms. Then in Section 5, we
prove that the blow-up set ${\bf q}$ is the critical point of $G_1^*$ and $G_2^*$, and further improve the estimate in Section 4, which leads to the proof of Theorem \ref{main1}.

\section{Proof of Theorem \ref{thm2}}
In the beginning of this section, we present some properties of the Green function defined on a flat torus  $\Omega=\mathbb{R} \slash \mathbb{Z}\omega_1 \times \mathbb{R} \slash \mathbb{Z}\omega_2 $. Lin and Wang\cite{LW1} used nonlinear partial differential equations to study the number of
critical points of Green function on flat torus.
\begin{thm}\cite{LW1}
The Green function has at most five critical points.
\end{thm}

Note that half periods  $\frac{1}{2}\omega_1$, $\frac{1}{2}\omega_2$ and $\frac{1}{2}(\omega_1+\omega_2)=\frac{1}{2}\omega_3$
are three obvious critical points of $G(x,0)$ and the other two critical points(non-trivial critical points) can be denoted as $\pm a$.

Denote
\begin{equation}
\mathcal{D}(q)=e^{-8\pi G(q,0)}\left( \int_{\Omega}\frac{ e^{8\pi(\gamma(z,q)-\gamma(q,q)+G(q,0)-G(z,0))}-1}{|z-q|^4}                             -\int_{\mathbb{R}^2\setminus \Omega}\frac{1}{|z-q|^4} \right)
\end{equation}

By \cite{LW2} and \cite{CKLW}, we know the sign of $\mathcal{D}(q)$.

\begin{thm}\label{thm53} \cite{LW2,CKLW}
Suppose $G(x,0)$ has extra critical points, then
$$\mathcal{D}(\frac{\omega_k}{2})>0, \,\,k=1,2,3,$$
and
$$\mathcal{D}(\text{extra critical point})=0.$$
\end{thm}
 ({\bf Proof of Theorem \ref{thm2}})
Let  $(u_{1,\varepsilon},u_{2,\varepsilon})$ be  a Liouville type fully bubbling solution of
\begin{equation}
\begin{cases}
\Delta u_{1,\varepsilon}+\frac{1}{\varepsilon^2}e^{u_{2,\varepsilon}}(1-e^{u_{1,\varepsilon}})=8\pi \delta_{p_1}\\
\Delta u_{2,\varepsilon}+\frac{1}{\varepsilon^2}e^{u_{1,\varepsilon}}(1-e^{u_{2,\varepsilon}})=8\pi \delta_{p_2}
\end{cases}
\end{equation}
and $p_1-p_2\in\{\frac{\omega_k}{2}  \}_{k=1,2,3}$. We assume that $q$ is the blow-up point. We will prove this theorem by contradiction, so assume the critical points of $G(x,0)$ are $\{\frac{\omega_k}{2}, \pm a \}_{k=1,2,3}.$   By Theorem \ref{main1}, $q$ is a critical point of $G(x,p_1)$ and $G(x, p_2)$, i.e.
\begin{equation}
\nabla G(q,p_1)=\nabla G(q-p_1,0)=0\quad\text{   and  }\quad \nabla G(q,p_2)=\nabla G(q-p_2,0)=0.
\end{equation}
Thus,  $q-p_1\in\{\frac{\omega_k}{2}, \pm a \}_{k=1,2,3}$ and  $q-p_2\in\{\frac{\omega_k}{2}, \pm a \}_{k=1,2,3}$. We will discuss all possible cases and  obtain a contradiction.\par
 Note that
\begin{equation}
\mathcal{D}^{(2)}(q)=\frac{\mathcal{D}(q-p_1)}{e^{-8\pi G(q,p_1)}}+\frac{\mathcal{D}(q-p_2)}{e^{-8\pi G(q,p_2)}}.
\end{equation}
Thus, if one of $q-p_1$ and $q-p_2$ belongs to $\{\frac{\omega_k}{2} \}_{k=1,2,3}$, Theorem \ref{thm53} implies
$\mathcal{D}^{(2)}(q)>0$, which is a contradiction to Theorem \ref{main1}.\par
Finally, if  $q-p_1$, $q-p_2\in \{\pm a\}$, we may assume
$q-p_1=a$ and $q-p_2=-a$ which implies $p_2-p_1=2a$. Since we assume $p_2-p_1\in \{\frac{\omega_k}{2}  \}_{k=1,2,3}$, thus
$a\in  \{\frac{\omega_k}{4}  \}_{k=1,2,3}$ which is called {\it four torsion point}. By Theorem 1.1 of \cite{CKLW}, a four torsion point
cannot be a critical point of $G(x,0)$. We thus conclude that $G(x, 0)$ has exactly three critical points.

\section{Profile of the bubbling solutions}
In this section, we show the local uniform estimate for the bubbling solution. Here we recall that $(u_{1,\varepsilon},u_{2,\varepsilon})$ is a sequence of fully bubbling solutions of Liouville type to \eqref{e002}.

We assume that $\max\{u_{1,\varepsilon}(x),u_{2,\varepsilon}(x)\}$ is attained at $x_{j,\varepsilon}$ for $x$ near $q_j$
and set
\begin{equation}\label{scal-fac}
\mu_{j,\varepsilon}=\varepsilon^{-1}e^{\frac{1}{2}\beta_{j,\varepsilon}}.
\end{equation}
Here we recall that $\beta_{j,\epsilon}\to -\infty$ and $\mu_{j,\epsilon}\to \infty$ as $\epsilon\to 0$. Let $\bar u_{i,\epsilon}$ be defined as
$$\overline{u}_{i,\varepsilon}(x)=u_{i,\varepsilon}(\mu_{j,\varepsilon}^{-1}x+x_{j,\varepsilon})-\beta_{j,\varepsilon},\quad i=1,2.$$
Then we prove that $(\overline{u}_{1,\varepsilon}(x),\overline{u}_{2,\varepsilon}(x))$ can be accurately approximated by the solutions to the following Liouville system
\begin{equation}\label{e092}
\begin{cases}
\Delta V_{1,j,\varepsilon}+e^{u_{0,2}(x_{j,\varepsilon})}e^{V_{2,j,\varepsilon}}=0\\
\Delta V_{2,j,\varepsilon}+e^{u_{0,1}(x_{j,\varepsilon})}e^{V_{1,j,\varepsilon}}=0 \quad \mbox{ in }\quad \mathbb R^2, \\
\end{cases}
\end{equation}
and we use $(\tilde{M}_{1,j,\varepsilon},\tilde{M}_{2,j,\varepsilon})=(\int_{\mathbb R^2}e^{u_{0,2}(x_{j,\varepsilon})}e^{V_{2,j,\varepsilon}},\int_{\mathbb R^2} e^{u_{0,1}(x_{j,\varepsilon})}e^{V_{1,j,\varepsilon}}  )$ to denote the energy of $(V_{1,j,\epsilon},V_{2,j,\epsilon})$. Fix $\delta>0$, we set the local energy of $(u_{1,\varepsilon}, u_{2,\varepsilon})$
\begin{equation}
\begin{cases}
m_{1,j,\varepsilon}=\frac{1}{\varepsilon^2}\int_{B_{\delta}(x_{j,\varepsilon})} e^{u_{2,j,\varepsilon}+u_{0,2}}(1-e^{u_{1,j,\varepsilon}+u_{0,1}}),\\
 m_{2,j,\varepsilon}=\frac{1}{\varepsilon^2}\int_{B_{\delta}(x_{j,\varepsilon})} e^{u_{1,j,\varepsilon}+u_{0,1}}(1-e^{u_{2,j,\varepsilon}+u_{0,2}}),
\end{cases}
\end{equation}
and the partition energy of $(u_{1,\varepsilon}, u_{2,\varepsilon})$
\begin{equation}
\begin{cases}
M_{1,j,\varepsilon}=\frac{1}{\varepsilon^2}\int_{ \Omega_j} e^{u_{2,j,\varepsilon}+u_{0,2}}(1-e^{u_{1,j,\varepsilon}+u_{0,1}}),\\
M_{2,j,\varepsilon}=\frac{1}{\varepsilon^2}\int_{\Omega_j} e^{u_{1,j,\varepsilon}+u_{0,1}}(1-e^{u_{2,j,\varepsilon}+u_{0,2}}),
\end{cases}
\end{equation}
where $\Omega_j$ is defined in (\ref{e108}).
One crucial step in the proof of Theorem \ref{main1} is the closeness of $m_{i,j,\epsilon}$ and $\tilde M_{i,j,\epsilon}$ as well as $8\pi$ and $\tilde M_{i,j,\epsilon}$. It is easy to prove that $\tilde M_{i,j,\epsilon}$ satisfies

\begin{equation}\label{e090}  \frac{1}{\tilde{M}_{1,j,\varepsilon}}+\frac{1}{\tilde{M}_{2,j,\varepsilon}}=\frac{1}{4\pi}.
\end{equation}  But in order to obtain accurate estimate between $\bar u_{i,\epsilon}$ and $V_{i,j,\epsilon}$ we need precise estimate of $8\pi-\tilde M_{i,j,\epsilon}$ for each $i=1,...,k$. In this section we shall first establish
\begin{equation}\label{e094}
m_{i,j,\varepsilon}-8\pi=o(1),\quad i=1,2,\,\,j=1,\cdots,k,
\end{equation}
and eventually we shall prove
\begin{equation}\label{e093}
 \tilde{M}_{i,j,\varepsilon}-8\pi =O( \mu_{j,\varepsilon}^{-1}+e^{\frac{1}{2}\beta_{j,\varepsilon}}).
\end{equation}
It is also interesting to compare our system with  $SU(3)$ Toda system.
For the result (2) in Theorem \ref{main1}, it is not simply obtained  by the same argument in  $SU(3)$ Toda system\cite{LWZ2012}  because of different dimension of kernel space at an entire solution. Due to   the coupling term $e^{u_1+u_2}$ in \eqref{e001}, we only know that the blow-up set ${\bf q}$ is  the critical point of $G_1^*+G_2^*$(see Lemma \ref{lem32}). Another difficulty  arises from this fact when we estimate   the error term $(\eta_{1,j,\varepsilon},\eta_{2,j,\varepsilon})$(see \eqref{e060}).\par

\begin{prop}\label{pro21}
Let $(u_{1,\varepsilon},u_{2,\varepsilon})$ be a bubbling solution of \eqref{e002}, then
\begin{equation}\label{e057}
\left|u_{i,j,\varepsilon}-\beta_{j,\varepsilon}-\frac{m_{i,j,\varepsilon}}{2\pi}\ln\frac{1}{| \mu_{j,\varepsilon}(x-x_{j,\varepsilon}) |}\right|\leq C
\end{equation}
for $x\in B_{\delta}(x_{j,\varepsilon})\setminus B_{  \mu_{j,\varepsilon}^{-1} }(x_{\varepsilon})$. Furthermore,

\begin{equation}
 |m_{i,j,\varepsilon}-8\pi|=o(1) ,\,\,i=1,2,\,\,j=1,\cdots,k.
\end{equation}

\end{prop}

The proof of this proposition is quite long and starts with some preliminary lemmas.  We first investigate the behavior of $(u_{1,\varepsilon},u_{2,\varepsilon})$ away from the blow-up set $\{q_1,\cdots,q_k\}$.
Even though we assume $N_1=N_2=2k$, the following two lemmas are still true  for $(N_1,N_2)$ satisfing
$$\frac{N_1}{4k}+\frac{N_2}{4k}=1. $$

\begin{lem}\label{lem21}
For small $\theta >0$, it holds
\begin{equation}
u_{i,\varepsilon}-\frac{1}{|\Omega|}\int_{\Omega}u_{i,\varepsilon}\to \sum_{j=1}^k m_{i,j}^* G(q_j,x) \quad \mbox{in    }C^1(\Omega\setminus\cup_{j=1}^k B_{\theta}(q_j))
\end{equation}
as $\varepsilon\to 0$, where $m_{i,j}^* =\lim_{\varepsilon\to 0}m_{i,j,\varepsilon}$, $i=1,2,$ $j=1,\cdots,k.$
\end{lem}

\begin{proof} We only prove the case $i=1$.
By the assumptions on the fully bubbling solutions, we find for any $\theta\in(0,\delta)$,
\begin{equation}
\frac{1}{\varepsilon^2}e^{u_{2,\varepsilon}+u_{0,2}}(1-e^{u_{1,\varepsilon}+u_{0,1}})=o(1) \quad \mbox{in    }\Omega\setminus\cup_{j=1}^k B_{\theta}(q_j)
\end{equation}
So,
\begin{equation*}
\frac{1}{\varepsilon^2}\int_{B_{\theta}(q_j)}e^{u_{2,\varepsilon}+u_{0,2}}(1-e^{u_{1,\varepsilon}+u_{0,1}})=m_{1,j,\varepsilon}+o(1)=m_{1,j}^*+o(1).
\end{equation*}
By this and Green representation,  we see that for $x\in \Omega\setminus\cup_{j=1}^k B_{\theta}(q_j)$ as $\varepsilon\to 0$,
\begin{equation*}
\begin{split}
&u_{1,\varepsilon}(x)-\frac{1}{|\Omega|}\int_{\Omega}u_{1,\varepsilon}\\
=&\frac{1}{\varepsilon^2}\int_{\Omega} G(y,x)e^{u_{2,\varepsilon}+u_{0,2}}(1-e^{u_{1,\varepsilon}+u_{0,1}})\\
=&\sum_{j=1}^k m_{1,j}^* G(q_j,x)+ \frac{1}{\varepsilon^2}\sum_{j=1}^k\int_{B_{\theta}(q_i)}\Big( G(y,x)-G(q_j,x)\Big)e^{u_{2,\varepsilon}+u_{0,2}}(1-e^{u_{1,\varepsilon}+u_{0,1}}) +o(1)\\
=&\sum_{j=1}^k m_{1,j}^* G(q_j,x)+o(1)
\end{split}
\end{equation*}
where the following fact is used:
\begin{equation}\label{new-a}
\frac{1}{\varepsilon^2} \int_{B_{\theta}(q_j)}|y-q_j|e^{u_{2,\varepsilon}+u_{0,2}}(1-e^{u_{1,\varepsilon}+u_{0,1}})=o(1).
\end{equation}
Indeed, by $\frac{1}{\varepsilon^2} e^{u_{2,\epsilon}+u_{0,2}}(1-e^{u_{1,\epsilon}+u_{0,1}})\to 0$ away from blowup points and the fact that $|y-q_j|\to 0$ as $y\to q_j$, we see that (\ref{new-a}) holds easily.

Similarly we have
\begin{equation}\label{e066}
Du_{1,\varepsilon}(x)= \sum_{j=1}^k m_{1,j}^* DG(q_j,x)+o(1), \quad x\in \Omega\setminus\cup_{j=1}^k B_{\theta}(q_j).
\end{equation}
\end{proof}

The location of the blow-up point can be determined by the Pohozaev identity below.

\begin{lem}\label{lem32} For $j=1,\cdots, k$, it holds
\begin{equation}\label{e054}
\begin{split}
 &m_{1,j}^*\left(\frac{\partial u_{0,2}}{\partial x_h}\Big\vert_{x=q_j}+\sum_{l\not=k, 1\leq l\leq k}   m^*_{2,l} \frac{\partial G(q_l,q_j)}{\partial x_h}\right)\\
 +&m_{2,j}^*\left(\frac{\partial u_{0,1}}{\partial x_h}\Big\vert_{x=q_j}+\sum_{l\not=k, 1\leq l\leq k}   m_{1,l}^*  \frac{\partial G(q_l,q_j)}{\partial x_h} \right)=0 .
 \end{split}
\end{equation}
\end{lem}

\begin{proof}
Fix $j\in\{1,\cdots,k\}$ and  let
$$\overline{u}_{1,\varepsilon}=u_{1,\varepsilon}-\frac{N_1\pi|x-q_j|^2}{|\Omega|},\quad\overline{u}_{2,\varepsilon}=u_{2,\varepsilon}-\frac{N_2\pi|x-q_j|^2}{|\Omega|} . $$
Then $(\overline{u}_{1,\varepsilon},\overline{u}_{2,\varepsilon})$
satisfies
\begin{equation}\label{e041}
\begin{cases}
\Delta \overline{u}_{1,\varepsilon} +\frac{1}{\varepsilon^2} h_2(x)e^{\overline{u}_{2,\varepsilon}+u_{0,2}}(1- h_1(x)e^{\overline{u}_{1,\varepsilon}+u_{0,1}})=0\\
\Delta \overline{u}_{2,\varepsilon} +\frac{1}{\varepsilon^2} h_1(x)e^{\overline{u}_{1,\varepsilon}+u_{0,1}}(1- h_2(x)e^{\overline{u}_{2,\varepsilon}+u_{0,2}})=0
\end{cases}
\end{equation}
where $h_1(x)=e^{\frac{N_1\pi|x-q_j|^2}{|\Omega|}}$ and  $h_2(x)=e^{\frac{N_2\pi|x-q_j|^2}{|\Omega|}}$. For $h=1,2,$ we consider
 the following Pohozaev identity for \eqref{e041}, which is obtained by multiplying $D_h\bar u_{2,\epsilon}$ to the first equation and $D_h\bar u_{1,\epsilon}$ to the second equation and integration by parts:
\begin{equation}\label{e031}
\begin{split}
&\int_{\partial B_{\theta}(q_j)} \Big(  \left<\nu, D\overline{u}_{1,\varepsilon}\right>D_h \overline{u}_{2,\varepsilon}
+\left<\nu, D\overline{u}_{2,\varepsilon}\right>D_h \overline{u}_{1,\varepsilon}- \nu_h\left< D\overline{u}_{1,\varepsilon} ,D\overline{u}_{2,\varepsilon} \right>  \Big) \\
=&\frac{1}{\varepsilon^2}\int_{B_{\theta}(q_j)} \Big( e^{u_{2,\varepsilon}+u_{0,2}}(1-e^{u_{1,\varepsilon}+u_{0,1}})(D_h u_{0,2}+     \frac{2N_2\pi (x-q_j)_h }{|\Omega|})\Big.\\
& \qquad\qquad \quad \Big.  +e^{u_{1,\varepsilon}+u_{0,1}}(1-e^{u_{2,\varepsilon}+u_{0,2}})(D_h u_{0,1}+\frac{2N_1\pi (x-q_j)_h}{|\Omega|} )\Big)\\
& -\frac{1}{\varepsilon^2} \int_{\partial B_{\theta}(q_j)}   \nu_h ( e^{u_{1,\varepsilon}+u_{0,1}}+e^{u_{2,\varepsilon}+u_{0,2}}-e^{u_{1,\varepsilon}+u_{0,1}+u_{2,\varepsilon}+u_{0,2}}    )
\end{split}
\end{equation}
where $\nu$ stands for the outer normal vector and $\nu_h=<\nu,h>$.
Then by Lemma \ref{lem21} and (\ref{new-a}) the right hand side (RHS) of (\ref{e031}) tends to

\begin{equation} \label{e055}
 m_{1,j}^*D_h u_{0,2}(q_j)+m_{2,j}^* D_h u_{0,1}(q_j).
\end{equation}
In order to evaluate the left hand side of (\ref{e031}) we introduce the following functions:

\begin{equation}\label{e065}
F_i(x,q_j)=\frac{1}{|\Omega|}\int_{\Omega} u_{i,\varepsilon}+\sum_{l=1}^k M_{i,j,\varepsilon} G(x,q_l)-\frac{N_i\pi|x-q_j|^2}{|\Omega|},
\end{equation}
for $i=1,2$ and $j=1,\cdots,k$. It is easy to see that
$\Delta F_i(x,q_j)=0$ for $x\not=q_j$. By this, \eqref{e065} and \eqref{e066},  we obtain
\begin{equation}\label{e056}
\begin{split}
&\mbox{ LHS of    }\eqref{e031}\\
=& \int_{\partial B_{\theta}(q_j)}\left< \nu, D F_2(x,q_j)\right>D_h F_1(x,q_j) +\left< \nu, D F_1(x,q_j)\right>D_h F_2(x,q_j)\\
&-  \int_{\partial B_{\theta}(q_j)}\nu_h\left<  D F_1(x,q_j), D F_2(x,q_j)\right>+o_{\varepsilon}(1)\\
=& \int_{\partial B_{\theta}(q_j)}\left< \nu, D \sum_{l=1}^k m_{2,l}^* G(x,q_l)   )\right>D_h \sum_{l=1}^k m_{1,l}^* G(x,q_l)\\
  &\quad +\left< \nu, D \sum_{l=1}^k m_{1,l}^* G(x,q_l)\right>D_h \sum_{l=1}^k m_{2,l}^* G(x,q_l)\\
  &- \left< D \sum_{l=1}^k m_{1,l}^* G(x,q_l) , D \sum_{l=1}^k m_{2,l}^* G(x,q_l)\right>\nu_h+o_{\varepsilon}(1)+o_{\theta}(1)\\
=&  -\sum_{l\not=k, 1\leq l\leq k} (m_{1,j}^*m^*_{2,l} +m_{2,j}^* m_{1,l}^*   )\frac{\partial G(q_l,q_j)}{\partial x_h}+o_{\varepsilon}(1)+o_{\theta}(1).
\end{split}
\end{equation}

Thus,
\eqref{e054} follows from \eqref{e055} and \eqref{e056}.

\end{proof}

Next, we will show that $m^*_{i,j}=\lim_{\varepsilon\to 0} m_{i,j,\varepsilon}=8\pi$ when $N_1=N_2=2k$ for $i=1,2$, $j=1,\cdots,k$.

\begin{lem}\label{lem8pi}
For  $i=1,2$, $j=1,\cdots,k$, it holds
\begin{equation}\label{e107}
m^*_{i,j}=8\pi.
\end{equation}
\end{lem}

\begin{proof}
Set
$$v_{1,j,\varepsilon}(x)=u_{1,\varepsilon}(\mu_{j,\varepsilon}^{-1}x+x_{j,\varepsilon})-\beta_{j,\varepsilon},
v_{2,j,\varepsilon}(x)=u_{2,\varepsilon}(\mu_{j,\varepsilon}^{-1}x+x_{j,\varepsilon})-\beta_{j,\varepsilon}.$$
Then $(v_{1,j,\varepsilon},v_{2,j,\varepsilon})$ satisfies
\begin{equation}
\begin{cases}
 \Delta v_{1,j,\varepsilon}+ e^{v_{2,j,\varepsilon}+u_{0,2}(\mu_{j,\varepsilon}^{-1}x+x_{j,\varepsilon})}( 1-e^{\beta_{j,\varepsilon}}   e^{v_{1,j,\varepsilon}+u_{0,1}(\mu_{j,\varepsilon}^{-1}x+x_{j,\varepsilon})}             )=
 \frac{4\pi N_1 \mu_{j,\varepsilon}^{-2} }{ |\Omega|},\\
  \Delta v_{2,j,\varepsilon}+ e^{v_{1,j,\varepsilon}+u_{0,1}(\mu_{j,\varepsilon}^{-1}x+x_{j,\varepsilon})}( 1-e^{\beta_{j,\varepsilon}}   e^{v_{2,j,\varepsilon}+u_{0,2}(\mu_{j,\varepsilon}^{-1}x+x_{j,\varepsilon})}             )=\frac{4\pi N_2\mu_{j,\varepsilon}^{-2}}{ |\Omega|}.
\end{cases}
\end{equation}
Note that $\mu_{j,\varepsilon}^{-1}$, $e^{\beta_{j,\varepsilon}}\to 0$ as $\varepsilon\to 0$. Thus, we have
$$(v_{1,j,\varepsilon},v_{2,j,\varepsilon})\to (V_{1,j},V_{2,j})\quad \text{  in  }\quad C^1_{loc}(\mathbb{R}^2), $$ where $(V_{1,j},V_{2,j})$
satisfies
\begin{equation}
\begin{cases}
\Delta V_{1,j}+e^{u_{0,2}(q_j)}e^{V_{2,j}}=0, \\
\Delta  V_{2,j}+e^{u_{0,1}(q_j)}e^{V_{1,j}}=0,\\
\int_{\mathbb{R}^2} e^{u_{0,2}(q_j)}e^{V_{2,j}}=M_{1,j},\,\, \int_{\mathbb{R}^2} e^{u_{0,1}(q_j)}e^{V_{1,j}}=M_{2,j},
\end{cases}
\end{equation}
and
\begin{equation}\label{e106}
 m_{1,j}^*\ge  M_{1,j}\quad \text{  and    }\quad m_{2,j}^*\ge M_{2,j}.
\end{equation}
On the other hand, it was   proved by Chanillo-Kiessling\cite{CK1994} that
  $(M_{1,j},M_{2,j})$ satisfies the following  Pohozaev identity
\begin{equation}\label{Mij}
 M_{1,j}+M_{2,j}=\frac{M_{1,j} M_{2,j}}{4\pi}
 \end{equation} which implies
$M_{1,j}=\frac{4\pi M_{2,j}}{M_{2,j}-4\pi}$ and $M_{2,j}=\frac{4\pi M_{1,j}}{M_{1,j}-4\pi}$. Here we note that by standard potential theory it is easy to prove $M_{1,j},M_{2,j}>4\pi$. Now we write (\ref{Mij}) as
\begin{equation}\label{e104}
(M_{1,j}-8\pi)+(M_{2,j}-8\pi)=\frac{(M_{1,j}-8\pi)^2}{M_{1,j}-4\pi}=\frac{(M_{2,j}-8\pi)^2}{M_{2,j}-4\pi}.
\end{equation}
Since $(u_{1,\varepsilon},u_{2,\varepsilon})$ is fully bubbling, we have
\begin{equation}\label{e105}
\sum_{j=1}^k m_{1,j}^*=8k\pi \text{   and  }\sum_{j=1}^k m^*_{2,j}=8k\pi.
\end{equation}
Combing \eqref{e106}, \eqref{e104} and \eqref{e105}, we have
\begin{equation}
 0 \ge \sum_{i=1}^2\sum_{j=1}^k(M_{i,j}-m_{i,j}^*)
 =  \sum_{i=1}^2\sum_{j=1}^k( M_{i,j}-8\pi)
 = \frac{1}{2}\sum_{i=1}^2\sum_{j=1}^k \frac{ (M_{i,j}-8\pi)^2}{M_{i,j}-4\pi}.
 \end{equation}
We conclude that $M_{i,j}=8\pi$. Furthermore, by \eqref{e106} and \eqref{e105} again, we obtain
$m_{i,j}^*=8\pi$.

\end{proof}

\begin{lem}\label{lem24}
 For $i=1,2$, $j=1,\cdots, k$, and any small $\theta'>0$, there exists a constant $C>0$ such that

\begin{equation}\label{e032}
u_{i,\varepsilon}(x)-u_{i,\varepsilon}(x_{j,\varepsilon})\leq -(4-\theta')\ln\Big\vert 1+ \mu_{j,\varepsilon}  |x-x_{j,\varepsilon}|\Big\vert+C
\end{equation}
for $x\in B_{\delta}(x_{j,\varepsilon})$.
\end{lem}
\begin{proof} By symmetry we only prove the case $i=1$. Using the Green's representation formula of $u_{1,\epsilon}$, we have

\begin{equation}
\begin{cases}
u_{1,\varepsilon}(x)-\frac{1}{|\Omega|}\int_{\Omega}u_{1,\varepsilon}=\frac{1}{\varepsilon^2}\int_{\Omega} G(y,x)e^{u_{2,\varepsilon}+u_{0,2}}(1-e^{u_{1,\varepsilon}+u_{0,1}})\\
u_{2,\varepsilon}(x)-\frac{1}{|\Omega|}\int_{\Omega}u_{2,\varepsilon}=\frac{1}{\varepsilon^2}\int_{\Omega} G(y,x)e^{u_{1,\varepsilon}+u_{0,1}}(1-e^{u_{2,\varepsilon}+u_{0,2}})
\end{cases}
\end{equation}
Taking out the regular part of $G$ we have
\begin{equation}\label{e033}
\begin{split}
&u_{1,\varepsilon}(x)-u_{1,\varepsilon}(x_{ j,\varepsilon})\\
=&\frac{1}{ \varepsilon^2}\int_{\Omega} (G(y,x)-G(y,x_{ j,\varepsilon}))e^{u_{2,\varepsilon}+u_{0,2}}(1-e^{u_{1,\varepsilon}+u_{0,1}})\\
=&  \frac{1}{ 2\pi\varepsilon^2} \int_{B_{ \delta}(x_{ j,\varepsilon})} \ln \frac{|y-x_{ j,\varepsilon}|}{|y-x|}e^{u_{2,\varepsilon}+u_{0,2}}(1-e^{u_{1,\varepsilon}+u_{0,1}}) +O(1)
\end{split}
\end{equation}
Let $\tilde{u}_{i,\varepsilon}(z)=u_{i,\varepsilon}(\mu_{j,\varepsilon}^{-1}z+x_{ j,\varepsilon})$, $i=1,2.$ then \eqref{e033} becomes
\begin{equation}\label{e034}
\begin{split}
&\tilde{u}_{1,\varepsilon}(z)-\tilde{u}_{1,\varepsilon}(0)\\
=&  \frac{1}{2\pi} \int_{B_{ \mu_{j,\varepsilon}\delta}(0)} \ln \frac{|y |}{|y-z|}e^{\tilde{u}_{2,\varepsilon}+\tilde{u}_{0,2}-\beta_{j,\varepsilon}}(1-e^{\tilde{u}_{1,\varepsilon}+\tilde{u}_{0,1}}) +O(1).
\end{split}
\end{equation}
where $\tilde u_{0,1}$ is understood similarly.
Let $\theta>0$ be a small constant. By \eqref{e107}, there exists $R>0$ such that
\begin{equation}
 \frac{1}{2\pi} \int_{B_{R}(0)}  e^{\tilde{u}_{2,\varepsilon}+\tilde{u}_{0,2}-\beta_{j,\varepsilon}}(1-e^{\tilde{u}_{1,\varepsilon}+\tilde{u}_{0,1}})\geq 4-\theta .
\end{equation}
Next, we split $B_{ \mu_{j,\varepsilon}\delta}(0)$ into
$$
B_{\frac{|z|}{2}}(0) \cup B_{\frac{|z|}{2}}(z) \cup \Big( B_{ \mu_{j,\varepsilon}\delta}(0)\setminus(  B_{\frac{|z|}{2}}(0) \cup B_{\frac{|z|}{2}}(z)  )   \Big)
$$
For $R\leq |y|\leq \frac{|z|}{2}$, we find $\frac{|y|}{|y-z|}\leq 1$,  and  for $y\in  B_{ \mu_{j,\varepsilon}\delta}(0)\setminus(  B_{\frac{|z|}{2}}(0) \cup B_{\frac{|z|}{2}}(z)  ) $, we find  $| \ln \frac{|y|}{|y-z|}|\leq C$ for some $C>0$.
By these and $ e^{\tilde{u}_{2,\varepsilon}+\tilde{u}_{0,2}-\beta_{j,\varepsilon}}(1-e^{\tilde{u}_{1,\varepsilon}+\tilde{u}_{0,1}})>0$,  we obtain

\begin{equation}\label{e035}
\begin{split}
&\tilde{u}_{1,\varepsilon}(z)-\tilde{u}_{1,\varepsilon}(0)\\
\leq &  \frac{1}{2\pi} \int_{ B_R(0)  } \ln \frac{|y |}{|y-z|}e^{\tilde{u}_{2,\varepsilon}+\tilde{u}_{0,2}-\beta_{j,\varepsilon}}(1-e^{\tilde{u}_{1,\varepsilon}+\tilde{u}_{0,1}})\\
&+ \frac{1}{2\pi} \int_{ B_{\frac{|z|}{2}}(z)  } \ln \frac{|y |}{|y-z|}e^{\tilde{u}_{2,\varepsilon}+\tilde{u}_{0,2}-\beta_{j,\varepsilon}}(1-e^{\tilde{u}_{1,\varepsilon}+\tilde{u}_{0,1}}) +O(1)
\end{split}
\end{equation}

Let $\sigma>0$ be a small constant, then
\begin{equation}\label{e036}
\begin{split}
&   \int_{ B_{\frac{|z|}{2}}(z)  } \ln \frac{|y |}{|y-z|}e^{\tilde{u}_{2,\varepsilon}+\tilde{u}_{0,2}-\beta_{j,\varepsilon}}(1-e^{\tilde{u}_{1,\varepsilon}+\tilde{u}_{0,1}}) \\
\leq  & C  \int_{ B_{\frac{\sigma}{2}}(z)  } \ln \frac{|y |}{|y-z|}  + \ln\frac{{3}|z|}{\sigma} \int_{ B_{\frac{|z|}{2}}(z)  \setminus B_{\frac{\sigma}{2}}(z)  }  e^{\tilde{u}_{2,\varepsilon}+\tilde{u}_{0,2}-\beta_{j,\varepsilon}}(1-e^{\tilde{u}_{1,\varepsilon}+\tilde{u}_{0,1}})\\
=& o_{\sigma}(1)\ln |z|+ \ln\frac{{3}|z|}{\sigma} \int_{ B_{\frac{|z|}{2}}(z)  \setminus B_{\frac{\sigma}{2}}(z)  }  e^{\tilde{u}_{2,\varepsilon}+\tilde{u}_{0,2}-\beta_{j,\varepsilon}}(1-e^{\tilde{u}_{1,\varepsilon}+\tilde{u}_{0,1}}) ,
\end{split}
\end{equation}
where $C $ is a constant.
Furthermore, if $|z|>> R$, then
$$e^{\tilde{u}_{2,\varepsilon}(y)+\tilde{u}_{0,2}(y)-\beta_{j,\varepsilon}}=O(\frac{1}{|z|^{4+o(1)}})\quad \text{ for }\quad   y\in B_{\frac{|z|}{2}}(z) , $$
and 
$$\ln |y-z|=\ln |z|+o_{\frac{1}{|z|}}(1)\quad \text{ for  }\quad y\in B_{R}(0).$$
By these, \eqref{e034}, \eqref{e035} and \eqref{e036}, we conclude that
\begin{equation}
\tilde{u}_{1,\varepsilon}(z)-\tilde{u}_{1,\varepsilon}(0)\leq -(4-\theta )\ln |z|+O(1).
\end{equation}

\end{proof}

By this lemma, we could refine the estimate of \eqref{e032} as follows.

\begin{lem}\label{lem25} For $i=1,2$, and $j=1,\cdots,k$, it holds
\begin{equation}\label{e070}
u_{i,\varepsilon}(x)-u_{i,\varepsilon}(x_{j,\varepsilon})=-\frac{m_{i,j,\varepsilon}}{2\pi}\ln (1+\mu_{j,\varepsilon}|x-x_{j,\varepsilon}|)+O(1)
\end{equation}
for $x\in B_{\delta}(x_{j,\varepsilon})$.
\end{lem}

\begin{proof}
As in  \eqref{e033}, we find
\begin{equation}\label{e037}
\begin{split}
&u_{1,\varepsilon}(x)-u_{1,\varepsilon}(x_{j,\varepsilon})\\
=&\frac{1}{2\pi\varepsilon^2}\int_{B_{\delta}(x_{j,\varepsilon})} \ln\frac{|y-x_{j,\varepsilon}|}{|y-x|}  e^{ {u}_{2,\varepsilon}+ {u}_{0,2} }(1-e^{ {u}_{1,\varepsilon}+ {u}_{0,1}})+O(1)\\
\end{split}
\end{equation}
Set  $y-x_{j,\varepsilon}=z\mu_{j,\varepsilon}^{-1}$ and $\overline{x}_{j,\varepsilon}=\mu_{j,\varepsilon}(x-x_{j,\varepsilon})$. Then \eqref{e037} becomes

\begin{equation}\label{e038}
\begin{split}
& {u}_{1,\varepsilon}(x)- {u}_{1,\varepsilon}(x_{ j,\varepsilon})\\
=&  \frac{1}{2\pi} \int_{ B_{\mu_{j,\varepsilon}\delta}(0)  } \ln \frac{|z |}{|z-\overline{x}_{j,\varepsilon}|}e^{\tilde{u}_{2,\varepsilon}+\tilde{u}_{0,2}-\beta_{j,\varepsilon}}(1-e^{\tilde{u}_{1,\varepsilon}+\tilde{u}_{0,1}})+O(1)\\
=& \frac{1}{2\pi} \int_{ B_{\mu_{j,\varepsilon}\delta}(0)  } \ln \frac{1}{|z-\overline{x}_{j,\varepsilon}|}e^{\tilde{u}_{2,\varepsilon}+\tilde{u}_{0,2}-\beta_{j,\varepsilon}}(1-e^{\tilde{u}_{1,\varepsilon}+\tilde{u}_{0,1}})+O(1).
\end{split}
\end{equation}
where Lemma \ref{lem24} is used.
Since
\begin{equation}\label{add-1}
e^{\tilde u_{2,\epsilon}+\tilde u_{0,2}-\beta_{j,\epsilon}}(1-e^{\tilde u_{1,\epsilon}+\tilde u_{0,1}})=O(1+|z|)^{\theta-4},
\end{equation}
the last term in (\ref{e070}) is
$$-\frac 1{2\pi}\log |\bar x_{j,\epsilon}|+O(1),\quad \mbox{ if }\quad |\bar x_{j,\epsilon}|>1. $$
In fact one just considers
subregions like $|z|<|\bar x_{j,\epsilon}|/2$, $|z-\bar x_{j,\epsilon}|\le |\bar x_{j,\epsilon}|/2$, $|z|>2|\bar x_{j,\epsilon}|$, etc. Because of the fast decaying rate in (\ref{add-1}) the derivation of (\ref{e070}) is standard and is therefore omitted.
\end{proof}

With this lemma, we can further refine Lemma \ref{lem21} as follows. In the next section(see Lemma \ref{lemma53}), we will prove that $ \mu_{1,\varepsilon},\cdots,$ and $\mu_{k,\varepsilon}$ are comparable, so we still use $O(  \mu_{j,\varepsilon}^{-1})$ instead of $O(\sum_{j=1}^k \mu_{j,\varepsilon}^{-1})$.

\begin{lem} \label{lem26} For  $i=1,2$,  we have
\begin{equation}
u_{i,\varepsilon}-\frac{1}{|\Omega|}\int_{\Omega}u_{i,\varepsilon}= \sum_{j=1}^k m_{i,j,\varepsilon}  G(x_{j,\varepsilon},x)+O(\mu_{j,\varepsilon}^{-1+o(1)}) \quad \mbox{in    }C^1(\Omega\setminus\cup_{j=1}^k B_{\theta}(x_{j,\varepsilon})).
\end{equation}
\end{lem}

\begin{proof}
By Lemma \ref{lem25}, we know
\begin{equation}\label{e039}
u_{1,\varepsilon}(x)=u_{1,\varepsilon}(x_{j,\varepsilon})+\frac{m_{1,j,\varepsilon}}{2\pi}\ln \frac{1}{\mu_{j,\varepsilon}}+O(1)\quad\text{  on  }\quad \partial B_{\delta}(x_{j,\varepsilon})
\end{equation}
On the other hand, by Lemma \ref{lem21}, we find

\begin{equation}\label{e040}
u_{1,\varepsilon}(x)= \frac{1}{|\Omega|}\int_{\Omega}u_{1,\varepsilon}+O(1)   \quad\text{  on  }\quad  \Omega\setminus \cup_{j=1}^k B_{\theta}(x_{j,\varepsilon} )
\end{equation}
Thus, from \eqref{e039} and \eqref{e040}, we conclude that
\begin{equation}
u_{1,\varepsilon}(x)= u_{1,\varepsilon}(x_{j,\varepsilon}) +\frac{m_{1,j,\varepsilon}}{2\pi}\ln \frac{1}{\mu_{j,\varepsilon}}+O(1)\quad \mbox{ for   }\quad x\in \Omega\setminus\cup_{j=1}^k B_{\theta}(x_{j,\varepsilon} ),
\end{equation}
and similarly,
\begin{equation}
u_{2,\varepsilon}(x)= u_{2,\varepsilon}(x_{j,\varepsilon}) +\frac{m_{2,j,\varepsilon}}{2\pi}\ln \frac{1}{\mu_{j,\varepsilon}}+O(1)\quad \mbox{ for   }\quad x\in \Omega\setminus\cup_{j=1}^k B_{\theta}(x_{j,\varepsilon} ).
\end{equation}

As a  result, for $x\in \Omega\setminus\cup_{j=1}^k B_{\theta}(x_{j,\varepsilon})$ we use (\ref{scal-fac}) to obtain
\begin{equation*}
\begin{split}
&u_{1,\varepsilon}(x)-\frac{1}{|\Omega|}\int_{\Omega}u_{1,\varepsilon}\\
=&\frac{1}{\varepsilon^2}\int_{\Omega} G(y,x) e^{u_{2,\varepsilon}+u_{0,2}}(1-e^{u_{1,\varepsilon}+u_{0,1}})dy \\
=&\sum_{j=1}^k\bigg (\frac{1}{\varepsilon^2}\int_{B_{\delta}(x_{j,\varepsilon})} G(y,x) e^{u_{2,\varepsilon}+u_{0,2}}(1-e^{u_{1,\varepsilon}+u_{0,1}}) +O\left(    \mu_{j,\varepsilon}^{-2+o(1)}   \right) \bigg )\\
=& \sum_{j=1}^k \bigg (m_{1,j,\varepsilon} G(x_{j,\varepsilon},x)+ \frac{1}{\varepsilon^2}\int_{B_{\delta}(x_{j,\varepsilon})} (G(y,x)-G(x_{j,\varepsilon},x)) e^{u_{2,\varepsilon}+u_{0,2}}(1-e^{u_{1,\varepsilon}+u_{0,1}}) \\
&\qquad +O\left(    \mu_{j,\varepsilon}^{-2+o(1)}\right) \bigg )\\
=& \sum_{j=1}^k\bigg ( m_{1,j,\varepsilon} G(x_{j,\varepsilon},x)+ O\left( \frac{1}{\varepsilon^2}\int_{B_{\delta}(x_{j,\varepsilon})} |y-x_{j,\varepsilon}| e^{u_{2,\varepsilon}+u_{0,2}}     \right)+O\left(    \mu_{j,\varepsilon}^{-2+o(1)}\right) \bigg )\\
=& \sum_{j=1}^k\bigg (m_{1,j,\varepsilon} G(x_{j,\varepsilon},x)+O(\mu_{j,\varepsilon}^{-1+o(1)}) \bigg )
\end{split}
\end{equation*}

Similarly, $\nabla u_{1,\epsilon}(x)$ satisfies:
\begin{equation*}
\nabla u_{1,\varepsilon}(x)=-   \sum_{j=1}^k\bigg (m_{1,j,\varepsilon} \nabla G(x_{j,\varepsilon},x)+O(\mu_{j,\varepsilon}^{-1+o(1)})\bigg ) 
\end{equation*}
for $x\in \Omega\setminus\cup_{j=1}^k B_{\theta}(x_{j,\varepsilon})$.

\end{proof}

\section{Sharp estimate ($0<\tau\leq \frac{1}{2}$) }

In this section, we will prove a sharper estimate than \eqref{e057} in Proposition \ref{pro21}. Before we state the result of this section, we introduce the approximation solutions and their asymptotic behavior.
For $j=1,\cdots,k$, we use $(V_{1,j,\varepsilon}, V_{2,j,\varepsilon})$ to denote the global solutions of
\begin{equation}\label{e095}
\begin{cases}
\Delta V_{1,j,\varepsilon}+e^{u_{0,2}(x_{j,\varepsilon})}e^{V_{2,j,\varepsilon}}=0,\qquad  x\in\mathbb{R}^2,\\
\Delta V_{2,j,\varepsilon}+e^{u_{0,1}(x_{j,\varepsilon})}e^{V_{1,j,\varepsilon}}=0,\qquad  x\in\mathbb{R}^2,\\
V_{1,j,\varepsilon}(0)=u_{1,\varepsilon}(x_{j,\varepsilon})-\beta_{j,\varepsilon},\quad V_{2,j,\varepsilon}(0)=u_{2,\varepsilon}(x_{j,\varepsilon})-\beta_{j,\varepsilon}\\
max_{x\in\mathbb{R}^2}V_{1,j,\varepsilon}(x)=V_{1,j,\varepsilon}(0),\\
\tilde{M}_{1,j,\varepsilon}=\int_{\mathbb{R}^2}e^{u_{0,2}(x_{j,\varepsilon})}e^{V_{2,j,\varepsilon}}, \quad
\tilde{M}_{2,j,\varepsilon}=\int_{\mathbb{R}^2}e^{u_{0,1}(x_{j,\varepsilon})}e^{V_{1,j,\varepsilon}}.
\end{cases}
\end{equation}
The system (\ref{e095}) is an irreducible Liouville system whose classification can be found in \cite{CSW} and \cite{linzhang1}. By the classification theorem for
such Liouville systems, $(V_{1,j,\varepsilon}, V_{2,j,\varepsilon})$ is radially symmetric with respect to the origin.\par

For $|x| $ large, it is well-known that

\begin{equation}\label{e058}
\begin{cases}
V_{1,j,\varepsilon}(x)=-\frac{\tilde{M}_{1,j,\varepsilon}}{2\pi}\ln |x|+I_{1,j,\varepsilon}+O(|x|^{2-\frac{\tilde{M}_{2,j,\varepsilon}}{2\pi}}),\\
V_{2,j,\varepsilon}(x)=-\frac{\tilde{M}_{2,j,\varepsilon}}{2\pi}\ln |x|+I_{2,j,\varepsilon}+O(|x|^{2-\frac{\tilde{M}_{1,j,\varepsilon}}{2\pi}}),
\end{cases}
\end{equation}
where $\{I_{i,j,\epsilon}\}$ is a sequence of constants tending to a finite constant $I_{i,j}$ when $\epsilon\to 0$,
and
\begin{equation}\label{e059}
\begin{cases}
V_{1,j,\varepsilon}'(|x|)=-\frac{\tilde{M}_{1,j,\varepsilon}}{2\pi}\frac{1}{|x|}+ O(|x|^{1-\frac{\tilde{M}_{2,j,\varepsilon}}{2\pi}}),\\
V_{2,j,\varepsilon}'(|x|)=-\frac{\tilde{M}_{2,j,\varepsilon}}{2\pi}\frac{1}{|x|}+O(|x|^{1-\frac{\tilde{M}_{1,j,\varepsilon}}{2\pi}}).
\end{cases}
\end{equation}
Let $(\phi_1,\phi_2)$ be solutions of the linearized system of \eqref{e095}.
\begin{equation}
\begin{cases}
\Delta \phi_1+e^{u_{0,2}(x_{j,\varepsilon})+V_{2,j,\varepsilon}}\phi_2=0\\
\Delta \phi_2+e^{u_{0,1}(x_{j,\varepsilon})+V_{1,j,\varepsilon}}\phi_1=0
\end{cases}\quad \text{in}\quad \mathbb{R}^2.
\end{equation}
By Theorem 2.1 of \cite{ZL-JFA}, if  $|\phi_{i}(x)|\leq C(1+|x|)^{\tau}$ for all $x\in\mathbb{R}^2$ for some $\tau\in(0,1)$ and $\phi_1(0)=\phi_2(0)=0$, there exist $c_1$ and $c_2$ such that
\begin{equation}\label{e096}
\begin{pmatrix}
\phi_1\\ \phi_2
\end{pmatrix}
= c_1 \begin{pmatrix}
 V_{1,j,\varepsilon}'(r)\frac{x_1}{r}\\
 V_{2,j,\varepsilon}'(r) \frac{x_1}{r}
\end{pmatrix}+c_2 \begin{pmatrix}
 V_{1,j,\varepsilon}'(r)\frac{x_2}{r}\\
 V_{2,j,\varepsilon}'(r) \frac{x_2}{r}
\end{pmatrix},\quad x\in \mathbb{R}^2;
\end{equation}

If $|\phi_{i}(x)|\leq C$ for all $x\in\mathbb{R}^2$, there exist $c_0$, $c_1$ and $c_2$ such that
\begin{equation}\label{e097}
\begin{pmatrix}
\phi_1\\ \phi_2
\end{pmatrix}
= c_0 \begin{pmatrix}
rV_{1,j,\varepsilon}'(r)+2\\
rV_{2,j,\varepsilon}'(r)+2
\end{pmatrix}
+c_1\frac{x_1}{r} \begin{pmatrix}
 V_{1,j,\varepsilon}'(r)\\
 V_{2,j,\varepsilon}'(r)
\end{pmatrix}+c_2 \frac{x_2}{r}\begin{pmatrix}
 V_{1,j,\varepsilon}'(r)\\
 V_{2,j,\varepsilon}'(r)
\end{pmatrix}
\end{equation}
for $x\in \mathbb{R}^2$.

Without loss of generality, we may assume that $u_{1,\varepsilon}(x_{1,j,\varepsilon})=\beta_{j,\varepsilon}$.
Denote the error term $(\eta_{1,j,\varepsilon},\eta_{2,j,\varepsilon})$ as follows.
\begin{equation}\label{e060}
\begin{cases}
\eta_{1,j,\varepsilon}(x)=&u_{1,\varepsilon}(x)-\beta_{j,\varepsilon}-U_{1,j,\varepsilon}^*(x)- {M}_{1,j,\varepsilon}(\gamma(x,x_{j,\varepsilon})-\gamma(x_{j,\varepsilon},x_{j,\varepsilon}))
\\&-\sum_{l\not= j,1\leq l\leq k} {M}_{1,l,\varepsilon}(G(x,x_{l,\varepsilon})-G(x_{j,\varepsilon},x_{l,\varepsilon}))\\
\eta_{2,j,\varepsilon}(x)=&u_{2,\varepsilon}(x)-\beta_{j,\varepsilon}-U^*_{2,j,\varepsilon}(x)- {M}_{2,j,\varepsilon}(\gamma(x,x_{j,\varepsilon})-\gamma(x_{j,\varepsilon},x_{j,\varepsilon}))
\\&-\sum_{l\not= j,1\leq l\leq k} {M}_{2,l,\varepsilon}(G(x,x_{l,\varepsilon})-G(x_{j,\varepsilon},x_{l,\varepsilon}))\\
\end{cases}
\end{equation}
where $(U_{1,j,\varepsilon}^*(x),U_{2,j,\varepsilon}^*(x))=\left(V_{1,j,\varepsilon}(\mu_{j,\varepsilon}(x-x_{j,\varepsilon}^*)),V_{2,j,\varepsilon}(\mu_{j,\varepsilon}(x-x_{j,\varepsilon}^*))\right)$ and the point $x_{j,\varepsilon}^*$ is chosen to satisfy
\begin{equation}
\begin{split}
D U_{1,j,\varepsilon}^*(x_{j,\varepsilon})= & D u_{0,1}(x_{j,\varepsilon})+  D u_{0,2}(x_{j,\varepsilon})\\
&+  \sum_{l\not = j,,1\leq l\leq k}\big( 16\pi DG( x_{j,\varepsilon},x_{l,\varepsilon})-  {M}_{1,l,\varepsilon} DG(x_{j,\varepsilon},x_{l,\varepsilon}) \big).
\end{split}
\end{equation}

It is not difficult to see that
\begin{equation*}
|x_{j,\varepsilon}-x_{j,\varepsilon}^*|=O(\mu_{j,\varepsilon}^{-2}).
\end{equation*}
Thus, by this and  \eqref{e060}, we obtain
\begin{equation*}
\eta_{i,j,\varepsilon}(x_{j,\varepsilon})=V_{i,j,\varepsilon}(\mu_{j,\varepsilon}(x_{j,\varepsilon}-x_{j,\varepsilon}^*))-V_{i,j,\varepsilon}(0)= O(\mu_{j,\varepsilon}^{-2}), \quad i=1,2.
\end{equation*}
By the choice of $x_{j,\varepsilon}^*$ and Lemma \ref{lem32}, we have
\begin{equation*}
\begin{split}
D \eta_{1,j,\varepsilon}(x_{j,\varepsilon})&= 
Du_{1,\varepsilon}(x_{j,\varepsilon})- D U_{1,j,\varepsilon}^*(x_{j,\varepsilon})- \sum_{l\not = j,,1\leq l\leq k}  {M}_{1,l,\varepsilon} DG(x_{j,\varepsilon},x_{l,\varepsilon})\\
&=-\left(D u_{0,1}(x_{j,\varepsilon})+  D u_{0,2}(x_{j,\varepsilon}) +  \sum_{l\not = j,,1\leq l\leq k} 16\pi DG( x_{j,\varepsilon},x_{l,\varepsilon})\right)\\
&=O(\mu_{j,\varepsilon}^{-2}).
\end{split}
\end{equation*}
For $i=1,2$, $j=1,\cdots,k$, we set 
\begin{equation}\label{fij}
\begin{split}
f_{i,j,\varepsilon}(x)=& u_{0,i}(x)-u_{0,i}(x_{j,\varepsilon})+{M}_{i,j,\varepsilon}(\gamma(x,x_{j,\varepsilon})-\gamma(x_{j,\varepsilon},x_{j,\varepsilon}))\\
& +\sum_{l\not= j,1\leq l\leq k} {M}_{i,l,\varepsilon}(G(x,x_{l,\varepsilon})-G(x_{j,\varepsilon},x_{l,\varepsilon})) 
 \end{split}
\end{equation}
and 
\begin{equation}\label{fije}
H_{i,j,\varepsilon}(x,t)=e^{ f_{i,j,\varepsilon}(x)+t}-1. 
\end{equation}

For any function $\xi(x)$, we define
$$\tilde{\xi}(x)=\xi(\mu_{j,\varepsilon}^{-1}x+x_{j,\varepsilon}). $$
Thus, $(\tilde{\eta}_{1,j,\varepsilon} (x),\tilde{\eta}_{2,j,\varepsilon} (x))$ satisfies
\begin{equation}\label{e008}
\begin{cases}
\Delta  \tilde{\eta}_{1,j,\varepsilon}(x)+e^{{u}_{0,2}(x_{j,\varepsilon})+\tilde{U}_{2,j,\varepsilon}} \tilde{\eta}_{2,j,\varepsilon}=g_{1,j,\varepsilon}\\
\Delta  \tilde{\eta}_{2,j,\varepsilon}(x)+ e^{{u}_{0,1}(x_{j,\varepsilon})+\tilde{U}_{1,j,\varepsilon}  }\tilde{\eta}_{1,j,\varepsilon}=g_{2,j,\varepsilon} \\
\tilde{\eta}_{i,j,\varepsilon}(0)=O(\mu_{j,\varepsilon}^{-2}),\,i=1,2,\,\,\, D\tilde{\eta}_{1,j,\varepsilon}(0)=O(\mu_{j,\varepsilon}^{-2}).
\end{cases}
\end{equation}
where $g_{1,j,\varepsilon}$ and $g_{2,j,\varepsilon}$  are
\begin{equation}\label{e110}
\begin{cases}
g_{1,j,\varepsilon}= e^{-\beta_{j,\varepsilon}}e^{\sum_{i=1}^2 (\tilde{u}_{i,\varepsilon} +\tilde{u}_{i,0})}+e^{u_{0,2}(x_{j,\varepsilon})}(e^{ \tilde{U}_{2,j,\varepsilon}  }-e^{ \tilde{U}^*_{2,j,\varepsilon}  } )\tilde{\eta}_{2,j,\varepsilon}    \\
 \qquad\quad -e^{\tilde{U}^*_{2,j,\varepsilon} + {u}_{0,2}(x_{j,\varepsilon})}( \tilde{H}_{2,j,\varepsilon}(x,\tilde{\eta}_{2,j,\varepsilon})-\tilde{\eta}_{2,j,\varepsilon}) \\
g_{2,j,\varepsilon}= e^{-\beta_{j,\varepsilon}}e^{\sum_{i=1}^2 (\tilde{u}_{i,\varepsilon} +\tilde{u}_{i,0})}+  e^{u_{0,1}(x_{j,\varepsilon})}(e^{ \tilde{U}_{1,j,\varepsilon}  }-e^{ \tilde{U}^*_{1,j,\varepsilon}  } )\tilde{\eta}_{1,j,\varepsilon}   \\
\qquad\quad  -e^{\tilde{U}^*_{1,j,\varepsilon} + {u}_{0,1}(x_{j,\varepsilon})} (\tilde{H}_{1,j,\varepsilon}(x,\tilde{\eta}_{1,j,\varepsilon})-\tilde{\eta}_{1,j,\varepsilon})
\end{cases},
\end{equation}
and $(   \tilde{U}_{1,j,\varepsilon}(x), \tilde{U}_{2,j,\varepsilon}(x) )=(V_{1,j,\varepsilon} (x),V_{2,j,\varepsilon}(x ))$.

By Proposition \ref{pro21},  we find that
\begin{equation}\label{e005}
e^{-\beta_{j,\varepsilon}}e^{\sum_{i=1}^2 (\tilde{u}_{i,\varepsilon} +\tilde{u}_{i,0})}=O\left(\frac{e^{\beta_{j,\varepsilon}}}{(1+|x|)^{8+o(1)}}\right)
\mbox{    for    }|x|\leq \mu_{j,\varepsilon}\delta
\end{equation}

For $\tau>0$ and $\tau'>0$, we set
\begin{equation*}
N_{\varepsilon}=\sup_{|x|\leq \mu_{j,\varepsilon}\delta}\max\left\{\frac{|\tilde{\eta}_{1,j,\varepsilon}|}{\alpha_{\varepsilon}(1+|x|)^{\tau}} ,\frac{|\tilde{\eta}_{2,j,\varepsilon}|}{\alpha_{\varepsilon}(1+|x|)^{\tau}} \right\},
\end{equation*}
where $\alpha_{\varepsilon}=\varepsilon^{2\tau}e^{-\tau\beta_{j,\varepsilon}}+e^{\tau'\beta_{j,\epsilon}}=\mu_{j,\varepsilon}^{-2\tau}+e^{\tau'\beta_{j,\varepsilon}}$.  

The main goal in this section is to show that    $N_{\varepsilon}$ is bounded for $0<\tau\leq \frac{1}{2}$ and
$0<\tau'<1$, and in the next section we will improve the boundedness of $N_{\varepsilon}$  for $0<\tau<1$ and
$0<\tau'<1$.

\begin{prop}\label{pro31}
\begin{equation}\label{e004}
N_{\varepsilon}\leq C
\end{equation}
 for some constant C, $0<\tau\leq \frac{1}{2}$ and $0<\tau'<1$.
\end{prop}

We will prove \eqref{e004} by contradiction. Thus, we suppose that
\begin{equation*}
N_{\varepsilon}\to +\infty.
\end{equation*}
Set
$$
N^*_{\varepsilon}=\sup_{|x|\leq \mu_{j,\varepsilon}\delta}\sup_{|x|=|x'|}\max\left\{\frac{|\tilde{\eta}_{1,j,\varepsilon}(x)-\tilde{\eta}_{1,j,\varepsilon}(x')|}{\alpha_{\varepsilon}(1+|x|)^{\tau}} ,\frac{|\tilde{\eta}_{2,j,\varepsilon}(x)-\tilde{\eta}_{2,j,\varepsilon}(x')|}{\alpha_{\varepsilon}(1+|x|)^{\tau}} \right\},
$$
For simplicity, we set $x_{j,\varepsilon}=x_{\varepsilon}$. In the following lemma, we  show that the radial part is a dominant term.
\begin{lem}\label{le32} It holds
$$N_{\varepsilon}^*=o(N_{\varepsilon}).$$
\end{lem}
\begin{proof}
Suppose it is not true, then there exists $c_0>0$ such that  $N_{\varepsilon}^*\geq c_0 N_{\varepsilon}$. We may
assume there exist  $x_{\varepsilon}'$ and $x_{\varepsilon}''$ with $|x_{\varepsilon}'|=|x_{\varepsilon}''|$ such that
$$ N_{\varepsilon}^*= \frac{|\tilde{\eta}_{1,j,\varepsilon}(x_{\varepsilon}')-\tilde{\eta}_{1,j,\varepsilon}(x_{\varepsilon}'')|}{\alpha_{\varepsilon}(1+|x'_{\varepsilon}|)^{\tau}}  .$$
We   also assume that $x_{\varepsilon}'$ and $x_{\varepsilon}''$  are symmetric with respect to $x_1$-axis. Set
$$\omega_{i,j,\varepsilon}^*(x)=\tilde{\eta}_{i,j,\varepsilon}(x)-\tilde{\eta}_{i,j,\varepsilon}(x^-) , \,\, x^-=(x_1,-x_2),\,\,x_2>0,\, $$
and $$\omega_{i,j,\varepsilon}(x)=\frac{\omega_{i,j,\varepsilon}^*(x)}{\alpha_{\varepsilon}(1+x_2)^{\tau}}, \,\,i=1,2.$$
Let $x_{\varepsilon}^{**}$ satisfy
$$
 \omega_{1,j,\varepsilon}(x_{\varepsilon}^{**}  ) =N_{\varepsilon}^{**}=
\sup_{|x|\leq \mu_{j,\varepsilon}\delta,\, x_2>0} \max\left\{\frac{| {\omega}_{1,j,\varepsilon}^*(x)|}{\alpha_{\varepsilon}(1+x_2)^{\tau}} ,\frac{| {\omega}^*_{2,j,\varepsilon}(x)|}{\alpha_{\varepsilon}(1+x_2)^{\tau}} \right\}
$$
It is clear that $N_{\varepsilon}^{**}\geq N_{\varepsilon}^*$. We divide the proof of this lemma into the following steps.\\
\\
{\it Step 1.} We claim $|x_{\varepsilon}^{**}|\leq \frac{1}{2}\mu_{j,\varepsilon}\delta$.\\
We will prove this claim by contradiction.
By Lemma \ref{lem26}, we find that
$$
\begin{cases}
v_{1,j,\varepsilon}(x):=u_{1,\varepsilon}(x)-\frac{1}{|\Omega|}\int_{\Omega}u_{1,\varepsilon}-\sum_{l=1}^k m_{1,l,\varepsilon}G(x,x_{l,\varepsilon})=O(\mu_{j,\varepsilon}^{-1+o(1)})\\
v_{2,j,\varepsilon}(x):=u_{2,\varepsilon}(x)-\frac{1}{|\Omega|}\int_{\Omega}u_{2,\varepsilon}-\sum_{l=1}^k m_{2,l,\varepsilon}G(x,x_{l,\varepsilon})=O(\mu_{j,\varepsilon}^{-1+o(1)})
\end{cases}
$$
in $C^1( \Omega\setminus \cup_{l=1}^k B_{\delta}(x_{l,\varepsilon}))$.
Thus,
\begin{equation*}
\begin{split}
|\omega_{1,j,\varepsilon}^*(x_{\varepsilon}^{**})|=&|\tilde{\eta}_{1,j,\varepsilon}(x_{\varepsilon}^{**})-\tilde{\eta}_{1,j,\varepsilon}(x_{\varepsilon}^{**^-})|\\
\leq &| v_{1,j,\varepsilon}(\mu_{j,\varepsilon}^{-1}x_{\varepsilon}^{**}+x_{\varepsilon})- v_{1,j,\varepsilon}(\mu_{j,\varepsilon}^{-1}x_{\varepsilon}^{**^-}+x_{\varepsilon})  |+ O(\mu_{j,\varepsilon}^{-2+o(1)})\\
&+ | U^*_{1,j,\varepsilon}(\mu_{j,\varepsilon}^{-1}x_{\varepsilon}^{**} +x_{ \varepsilon})- U^*_{1,\varepsilon}(\mu_{j,\varepsilon}^{-1}x_{\varepsilon}^{**^-} +x_{\varepsilon})  | \\
\leq &C \left(  \mu_{j,\epsilon}^{-1+o(1)}+\mu_{j,\epsilon}^{-2+o(1)}|x_{\epsilon,2}^{**}|\right).
\end{split}
\end{equation*}
Since we assume $\frac{1}{2}\mu_{j,\varepsilon}\delta<|x_{\varepsilon}^{**}|\leq \mu_{j,\varepsilon}\delta$
and $0<\tau\leq\frac{1}{2}$, we obtain
\begin{equation}
\begin{split}
N_{\varepsilon}^{**}\leq  &\frac{\mu_{j,\varepsilon}^{-2+o(1)}|x_{\varepsilon,2}^{**}|^{1-\tau}}{\alpha_{\varepsilon}} +
\frac{\mu_{j,\varepsilon}^{-1+o(1)}|x_{\varepsilon,2}^{**}|^{ -\tau}}{\alpha_{\varepsilon}}\\
\leq& C\left(\frac{\mu_{j,\varepsilon}^{-2+o(1)}\mu_{j,\varepsilon}^{1-\tau}}{\alpha_{\varepsilon}} +
\frac{\mu_{j,\varepsilon}^{-1+o(1)}\mu_{j,\varepsilon}^{ -\tau}}{\alpha_{\varepsilon}}\right)=o(1),
\end{split}
\end{equation}
which is a contradiction to what we assumed of $N_{\epsilon}^{**}$.\\
\\
{\it Step 2.} We claim that $|x_{\varepsilon}^{**}|$ is bounded. \par
We prove this claim by contradiction.
After calculation, we have
\begin{equation}\label{e007}
\begin{split}
&|g_{1,j,\varepsilon}|\\
\leq& C\left(\frac{e^{\beta_{j,\varepsilon}}}{(1+|x|)^{8+o(1)}} +\frac{ \mu_{j,\varepsilon}^{-1}| \tilde{\eta}_{2,j,\varepsilon}(x)| + | \tilde{\eta}_{2,j,\varepsilon}(x)|^2+\mu_{j,\varepsilon}^{-1}|Df_{2,j,\varepsilon}(x_{j,\varepsilon})||x|+\mu_{j,\varepsilon}^{-2}|x|^2}{(1+|x|)^{4+o(1)}} \right)\\
\leq&C\left(    \frac{e^{\beta_{j,\varepsilon}}}{(1+|x|)^{8+o(1)}}  +\frac{\mu_{j,\varepsilon}^{-1}\alpha_{\varepsilon}N_{\varepsilon}}{(1+|x|)^{4-\tau+o(1)}}  +\frac{(\alpha_{\varepsilon}N_{\varepsilon})^2}{(1+|x|)^{4-2\tau+o(1)}}+\frac{\mu_{j,\varepsilon}^{-1}|Df_{2,j,\varepsilon}(x_{j,\varepsilon})| |x|+\mu_{j,\varepsilon}^{-2}|x|^2}{(1+|x|)^{4+ o(1)}}         \right)
\end{split}
\end{equation}
where $f_{i,j,\epsilon}$ is defined in (\ref{fij}).
Moreover we have
\begin{align*}
&\Delta \omega_{1,j,\varepsilon}+2\tau\frac{1}{1+x_2}\frac{\partial \omega_{1,j,\varepsilon}}{\partial x_2}-\left(\frac{\tau(1-\tau)}{(1+x_2)^2}\right)\omega_{1,j,\varepsilon}+e^{u_{0,2}(x_{\varepsilon})+\tilde{U}^*_{2,j,\varepsilon}}\omega_{2,j,\varepsilon}\\
=&\frac{g_{1,j,\varepsilon}(x^-)-g_{1,j,\varepsilon}(x) }{\alpha_{\varepsilon}(1+x_2)^{\tau}}.
\end{align*}
Recall that $D \omega_{1,j,\varepsilon}(x_{\varepsilon}^{**})=0$ and $\Delta \omega_{1,j,\varepsilon}(x_{\varepsilon}^{**})\leq 0$.
If $|x_{\varepsilon}^{**}|\geq R$ for some large $R>0$,
then
\begin{equation*}
\frac{\omega_{1,j,\varepsilon}(x_{\varepsilon}^{**})}{(1+x_{\varepsilon,2}^{**})^2}\leq  C\left( \frac{g_{1,j,\varepsilon}(x_{\varepsilon}^{**})-g_{1,j,\varepsilon}(x_{\varepsilon}^{**^-}) }{\alpha_{\varepsilon}(1+x_{\varepsilon,2}^{**})^{\tau}}
+\frac{|\omega_{2,j,\varepsilon}(x_{\varepsilon}^{**})|}{(1+|x_{\varepsilon}^{**}|)^{4+o(1)}}\right).
\end{equation*}
By this and \eqref{e007}, we obtain

\begin{equation}\label{re1}
\begin{split}
&\alpha_{\varepsilon}\omega_{1,j,\varepsilon}(x_{\varepsilon}^{**})\\ 
\leq&C(1+|x_{\varepsilon}^{**}|)^{2-\tau}\left(    \frac{e^{\beta_{j,\varepsilon}}}{(1+|x_{\varepsilon}^{**}|)^{8+o(1)}}  +\frac{\mu_{j,\varepsilon}^{-1}\alpha_{\varepsilon}N_{\varepsilon}}{(1+|x_{\varepsilon}^{**}|)^{4-\tau+o(1)}} +\frac{(\alpha_{\varepsilon}N_{\varepsilon})^2}{(1+|x_{\varepsilon}^{**}|)^{4-2\tau+o(1)}}\right. \\
&\left. +\frac{\mu_{j,\varepsilon}^{-1}|Df_{2,j,\varepsilon}(x_{j,\varepsilon})| |x_{\varepsilon}^{**}|+\mu_{j,\varepsilon}^{-2}|x_{\varepsilon}^{**}|^2}{(1+|x_{\varepsilon}^{**}|)^{4+ o(1)}}      +\frac{\alpha_{\varepsilon}N_{\varepsilon}}{(1+|x_{\varepsilon}^{**}|)^{4-\tau+o(1)}}   \right)
\end{split}
\end{equation}
Note that $\alpha_{\varepsilon}N_{\varepsilon}=o(1)$.  We deduce from \eqref{re1} that
\begin{equation}\label{e112}
N_{\varepsilon}^{**}\leq \frac{1}{\alpha_{\varepsilon}}(\mu_{j,\varepsilon}^{-1}|D f_{2,j,\varepsilon}(x_{\varepsilon})|+e^{\beta
_{j,\varepsilon}} +\mu_{j,\varepsilon}^{-2}   )=O(1)
\end{equation}
provided $0<\tau\leq \frac{1}{2}$ and $|x_{\varepsilon}^{**}|$ is large. It is a contradiction to $N_{\epsilon}^*\ge c_0N_{\epsilon}\to \infty$. \\
\\
{\it Step 3. } Let $\omega_{i,j,\varepsilon}^{**}=\frac{\omega_{i,\varepsilon}}{N_{\varepsilon}^{**}}$, $i=1,2.$  Then
$\frac{g_{i,\varepsilon}}{\alpha_{\varepsilon}N_{\varepsilon}^{**}(1+x_2)^{\tau}}\to 0$, $i=1,2.$
So, $((1+x_2)^{\tau}\omega_{1,j,\varepsilon}^{**},(1+x_2)^{\tau}\omega_{2,j,\varepsilon}^{**} )\to (\omega_1,\omega_2)\not=(0,0)$
in any compact set of $\mathbb{R}^2$, and $(\omega_1,\omega_2)$
satisfies
\begin{equation}\label{re2}
\begin{cases}
\Delta \omega_1+e^{u_{0,2}(q)+U_2}\omega_2=0\\
\Delta \omega_2+e^{u_{0,1}(q)+U_1}\omega_1=0\\
\omega_1(x_1,0)=\omega_{2}(x_1,0)=0.
\end{cases}
\end{equation}
On the other hand, we deduced from $D\tilde{\eta}_{1,j,\varepsilon}(0)=O(\mu_{j,\varepsilon}^{-2})$ that $D \omega_1(0)=0$.
Together with \eqref{re2}, we obtain $\omega_1=\omega_2\equiv 0.$ This is a contradiction. We conclude that
$$N_{\varepsilon}^*=o(N_{\varepsilon}).$$
\end{proof}
(Proof of Proposition \ref{pro31} )
Denote
$$\varphi_{i,j,\varepsilon}=\frac{1}{2\pi}\int_{0}^{2\pi}\tilde{\eta}_{i,j,\varepsilon}(r,\theta)d\theta,\,\, r=|x-x_{\varepsilon}|,\,\,i=1,2.$$
By Lemma \ref{le32},  we conclude that
\begin{equation}\label{e098}
\sup_{r\leq \mu_{j,\varepsilon}^{-1}\delta} \max \left\{\frac{\varphi_{1,j,\varepsilon}}{(1+r)^{\tau}} ,\frac{\varphi_{2,j,\varepsilon}}{(1+r)^{\tau}}   \right\}=\alpha_{\varepsilon}N_{\varepsilon}(1+o(1)).
\end{equation}
Without loss of generality, we assume that $ \frac{\varphi_{1,j,\varepsilon}}{(1+r)^{\tau}} $ attains the maximum of
$$\sup_{r\leq \mu_{j,\varepsilon}^{-1}\delta} \max \left\{\frac{\varphi_{1,j,\varepsilon}}{(1+r)^{\tau}} ,\frac{\varphi_{2,j,\varepsilon}}{(1+r)^{\tau}}   \right\}$$
at $r_{\varepsilon}$.
Suppose $r_{\varepsilon}$ is bounded. By integrating \eqref{e008}, we obtain
\begin{equation}\label{e099}
\begin{cases}
\Delta \varphi_{1,j,\varepsilon}+e^{u_{0,2}(x_{\varepsilon})+V_{2,j,\varepsilon}}\varphi_{2,j,\varepsilon}=h_{1,j,\varepsilon},\\
\Delta \varphi_{2,j,\varepsilon}+e^{u_{0,1}(x_{\varepsilon})+V_{1,j,\varepsilon}}\varphi_{1,j,\varepsilon}=h_{2,j,\varepsilon},
\end{cases}
\end{equation}
where $h_{i,j,\varepsilon}=\frac{1}{2\pi}\int_{0}^{2\pi} g_{i,j,\varepsilon}(r,\theta)d\theta$, $r=|x-x_{\varepsilon}|$.
Set
$$\varphi_{i,j,\varepsilon}^*(x)=\frac{\varphi_{i,j,\varepsilon}(|x|)}{\varphi_{1,j,\varepsilon}(r_{\varepsilon})},\quad i=1,2.$$
By \eqref{e098}, we have
$$|\varphi_{i,j,\varepsilon}^*(x)|\leq \frac{C|\varphi_{i,j,\varepsilon}|}{\alpha_{\varepsilon}N_{\varepsilon}} \leq C(1+|x|)^{\tau},  $$
and   by \eqref{e007}
\begin{equation}\label{re4}
\begin{split}
 &\frac{|h_{1,j,\varepsilon}(r)|}{|\varphi_{1,j,\varepsilon}(r_{\varepsilon})|  } \leq \frac{C|h_{1,j,\varepsilon}(r)|}{\alpha_{\varepsilon}N_{\varepsilon}}\\
  \leq& C\left(\frac{1}{N_{\varepsilon}}+\mu_{j,\varepsilon}^{-1}+\alpha_{\varepsilon}N_{\varepsilon}   +\frac{\mu_{j,\varepsilon}^{-1}|Df_{2,j,\varepsilon}(x_{\varepsilon})| +\mu_{j,\varepsilon}^{-2} }{\alpha_{\varepsilon}N_{\varepsilon}}\right).
\end{split}
\end{equation}
Note that $\alpha_{\varepsilon}N_{\varepsilon}=o(1)$.  Thus, $(\varphi^*_{1,j,\varepsilon},\varphi^*_{2,j,\varepsilon})\to (\varphi_1,\varphi_2)\not=(0,0)$ in any compact subset of $\mathbb{R}^2$ and $(\varphi_1,\varphi_2)$ satisfies
\begin{equation}
\begin{cases}
\Delta \varphi_1+e^{u_{0,2}(q_j)+V_{2,j}}\varphi_2=0,\\
\Delta \varphi_2+e^{u_{0,2}(q_j)+V_{1,j}}\varphi_1=0.
\end{cases}
\end{equation}
But $\varphi^*_{i,j,\varepsilon}(0)=O(\frac{\mu_{j,\varepsilon}^{-2}}{\alpha_{\varepsilon}N_{\varepsilon}})=o(1),\,\,i=1,2,$ and thus
$\varphi_1(0)=\varphi_2(0)=0$, which implies $(\varphi_1,\varphi_2)\equiv (0,0)$, a contradiction to the fact that at least one of
$\varphi^*_{1,j,\epsilon}(r_{\epsilon}) $ and $\varphi^*_{2,j,\epsilon}(r_{\epsilon}) $
is $1$.

Next, assume that $r_{\varepsilon}\to \infty$. After calculation, we obtain
\begin{equation*}
r\varphi_{1,j,\varepsilon}'(r)\leq C\left(  {\alpha_{\varepsilon}N_{\varepsilon}} + e^{\beta_{j,\varepsilon}} +\mu_{j,\varepsilon}^{-2}+ {\mu_{j,\varepsilon}^{-1}|Df_{2,j,\varepsilon}(x_{\varepsilon})|} + {\mu_{j,\varepsilon}^{-2}\ln (1+r)} \right)
\end{equation*}
for $r>1$. Let $r_0$ be a large constant. Then
\begin{equation*}
\begin{split}
\alpha_{\varepsilon}N_{\varepsilon}(1+r_{\varepsilon})^{\tau}&=|\varphi_{1,j,\varepsilon}(r_{\varepsilon})|
\leq \int_{r_0}^{r_{\varepsilon}}|\varphi_{1,j,\varepsilon}'(r) |dr\\
&\leq C\left(
( {\alpha_{\varepsilon}N_{\varepsilon}} + e^{\beta_{j,\varepsilon}} +\mu_{j,\varepsilon}^{-2}+ {\mu_{j,\varepsilon}^{-1}|Df_{2,j,\varepsilon}(x_{\varepsilon})|})\ln r_{\varepsilon}+\mu_{j,\varepsilon}^{-2}(\ln r_{\varepsilon})^{2}\right),
\end{split}
\end{equation*}
from which, we deduce that 
\begin{equation}\label{e113}
N_{\varepsilon}\leq C\left(\frac{ e^{\beta_{j,\varepsilon}} +\mu_{j,\varepsilon}^{-2}+ {\mu_{j,\varepsilon}^{-1}|Df_{2,j,\varepsilon}(x_{\varepsilon})|} +\mu_{j,\varepsilon}^{-2}}{\alpha_{\varepsilon}}\right)=O(1)
\end{equation}
Again, it is a contradiction. We conclude that $N_{\varepsilon}$ is bounded.

\section{Sharp estimate($0<\tau\leq 1$) and Necessary condition}
By Proposition \ref{pro31}, we have

\begin{prop}\label{pro41}
For $\tau\in(0,\frac{1}{2}]$ and $\tau'\in(0,1)$, it holds
\begin{equation*}
| {\eta}_{i,j,\varepsilon}(x)|\leq C(1+ \mu_{j,\varepsilon}|x-x_{j,\varepsilon}|)^{\tau}( \mu_{j,\varepsilon}^{-2\tau}+e^{\tau'\beta_{j,\varepsilon}}), \,\,x\in B_{ \delta}(x_{j,\varepsilon}),\,\,i=1,2 ,\,\, j=1,\cdots,k.
\end{equation*}
\end{prop}
With this proposition and symmetry, we are led to a refined estimate for the radial part of $g_{i,j,\varepsilon}$, $h_{i,j,\varepsilon}$.
  Note that
\begin{equation}\label{e014}
H_{2,j,\varepsilon }(x,\eta_{2,j,\varepsilon})-\eta_{2,j,\varepsilon}=H_{2,j,\varepsilon}(x,0)+H_{2,j,\varepsilon}(x,0)\eta_{2,j,\varepsilon}+O(|\eta_{2,j,\varepsilon}|^2)\end{equation}
and
\begin{equation}\label{e015}\begin{split}
H_{2,j,\varepsilon}(x,0)=&Df_{2,j,\varepsilon}(x_{j,\varepsilon})\cdot (x-x_{j,\varepsilon})+\frac{1}{2}\sum_{a,b=1,2}D^2f_{2,j,\varepsilon}(x-x_{j,\varepsilon})_a(x-x_{j,\varepsilon})_b\\
&+|D f_{2,j,\varepsilon}(x_{j,\varepsilon})|^2|x-x_{j,\varepsilon}|^2+O(|x-x_{j,\varepsilon}|^3)
\end{split}
\end{equation}
By  \eqref{e104}, \eqref{e105} and symmetry, we have
\begin{equation}\label{e111}
\begin{split}
h_{1,j,\varepsilon}=&\frac{1}{2\pi} \int_{0}^{2\pi}g_{1,j,\varepsilon}\\
=&-\frac{1}{2\pi} \int_0^{2\pi}e^{\tilde{U}^*_{2,j,\varepsilon} + {u}_{0,2}(x_{j,\varepsilon})}( \tilde{H}_{2,j,\varepsilon}(x,\tilde{\eta}_{2,j,\varepsilon})-\tilde{\eta}_{2,j,\varepsilon})d\theta \\
&+\frac{1}{2\pi} \int_0^{2\pi}\left(e^{u_{0,2}(x_{j,\varepsilon})}(e^{ \tilde{U}_{2,j,\varepsilon}  }-
e^{ \tilde{U}^*_{2,j,\varepsilon}  } )\tilde{\eta}_{2,j,\varepsilon}   + e^{-\beta_{\varepsilon}}e^{\sum_{i=1}^2 (\tilde{u}_{i,\varepsilon} +\tilde{u}_{i,0})}\right)d\theta \\
 =& -\frac{1}{2\pi} \int_0^{2\pi}e^{\tilde{U} _{2,j,\varepsilon} + {u}_{0,2}(x_{j,\varepsilon})}\left( |Df_{2,j,\varepsilon}(x_{j,\varepsilon})|^2\mu_{j,\varepsilon}^{-2}r^2 + |\tilde{\eta}_{2,j,\varepsilon}|^2\right)d\theta \\
 &+O\left( \mu_{j,\varepsilon} |x_{j,\varepsilon}-x_{j,\varepsilon}^*|  \int_0^{2\pi} e^{\tilde{U} _{2,j,\varepsilon} + {u}_{0,2}(x_{j,\varepsilon})}  |\tilde{\eta}_{2,j,\varepsilon}|   d\theta \right)\\
&+O\left(\frac{e^{\beta_{j,\varepsilon}}}{(1+r)^{8+o(1)}}\right) \\
=&O\left(  \frac{(\mu_{j,\varepsilon}^{-2}+e^{\beta_{j,\varepsilon}})}{(1+r)^{2+o(1)}} \right),\end{split}
\end{equation}
where $r=|x-x_{j,\varepsilon}|$.
Then, using the same argument of Lin-Zhang(page 2608 of \cite{ZL-JFA}), we have an estimate for the radial part of $\tilde{\eta}_{i,j,\varepsilon}$.
\begin{prop}\label{pro42}
It holds
\begin{equation}\label{e100}
| {\varphi}_{i,j,\varepsilon}(r)|\leq C(\mu_{j,\varepsilon}^{-2}+e^{\beta_{j,\varepsilon}})(1+r)^{\delta_1},\qquad    r\leq \mu_{j,\varepsilon}\delta
\end{equation}

for some small $\delta_1>0$ and $i=1,2 ,\,\, j=1,\cdots,k.$
\end{prop}

By Proposition \ref{pro41} and Proposition \ref{pro42}, we have  an estimate for the  convergence rate of   $\tilde{M}_{i,j,\varepsilon}$ to $8\pi$.

\begin{lem}\label{lem42} For $i=1,2,$ and $j=1,\cdots,k,$
we have
\begin{equation}\label{e053}
\tilde{M}_{i,j,\varepsilon }-m_{i,j,\varepsilon}=O\left(\mu_{j,\varepsilon}^{-2}\ln \mu_{j,\varepsilon}\sum_{i=1,2}|D f_{i,j,\varepsilon}(x_{j,\varepsilon})|^2 +\mu_{j,\varepsilon}^{-2}+e^{\beta_{j,\varepsilon}} \right).
\end{equation}
and \begin{equation}\label{e076}
 \tilde{M}_{i,j,\varepsilon }-8\pi  =O\left(|\tilde{M}_{i,j,\varepsilon }-m_{i,j,\varepsilon}|^{\frac{1}{2}}\right).
\end{equation}
\end{lem}

\begin{proof} For convenience, we only prove for the case of $i=1$.\\
   By Proposition \ref{pro41}, we find
\begin{equation}\label{e061}
\begin{split}
&\tilde{M}_{1,j,\varepsilon}-m_{1,j,\varepsilon}\\
=&\frac{1}{\varepsilon^2} \int_{\mathbb{R}^2} e^{u_{0,2}(x_{j,\varepsilon})+\beta_{j,\varepsilon}+U_{2,j,\varepsilon}^*}-\frac{1}{\varepsilon^2}\int_{B_{\delta}(x_{j,\varepsilon})} e^{u_{2,\varepsilon}+u_{0,2}}(1-e^{u_{1,\varepsilon}+u_{0,1}})  \\
= &-\frac{1}{\varepsilon^2}\int_{B_{\delta}(x_{j,\varepsilon})} e^{u_{0,2}(x_{j,\varepsilon})+\beta_{j,\varepsilon} +U_{2,j,\varepsilon}^*  }H_{2,j,\varepsilon}(x,\eta_{2,j,\varepsilon}) + \frac{1}{\varepsilon^2}\int_{\mathbb{R}^2\setminus B_{\delta}(x_{j,\varepsilon})} e^{u_{0,2}(x_{j,\varepsilon})+\beta_{j,\varepsilon} +U_{2,j,\varepsilon}^*  }  \\
&+ \frac{1}{\varepsilon^2}\int_{B_{\delta}(x_{j,\varepsilon})  } e^{\sum_{i=1}^2 (u_{i,\varepsilon}+u_{0,i})}\\
=& -\frac{1}{\varepsilon^2}\int_{B_{\delta}(x_{j,\varepsilon})} e^{u_{0,2}(x_{j,\varepsilon})+\beta_{j,\varepsilon} +U_{2,j,\varepsilon}^*  }H_{2,j,\varepsilon}(x,\eta_{2,j,\varepsilon})+O(\mu_{j,\varepsilon}^{-2+o(1)})+O(e^{\beta_{j,\varepsilon}}).
\end{split}
\end{equation}
 As in \eqref{e111}, we obtain
\begin{equation}\label{e062}
\begin{split}
&\frac{1}{\varepsilon^2}\int_{B_{\delta}(x_{j,\varepsilon)}} e^{u_{0,2}(x_{j,\varepsilon})+\beta_{j,\varepsilon} +U_{2,j,\varepsilon}^*}H_{2,j,\varepsilon}(x,\eta_{2,j,\varepsilon})\\
=&\frac{1}{\varepsilon^2}\int_{B_{\delta}(x_{j,\varepsilon)}} e^{u_{0,2}(x_{j,\varepsilon})+\beta_{j,\varepsilon} +U_{2,j,\varepsilon} }H_{2,j,\varepsilon}(x,\eta_{2,j,\varepsilon})\\
&+ \frac{1}{\varepsilon^2}\int_{B_{\delta}(x_{j,\varepsilon)}} e^{u_{0,2}(x_{j,\varepsilon})+\beta_{j,\varepsilon}  +U_{2,j,\varepsilon}}  (e^{U_{2,j,\varepsilon}^*-U_{2,j,\varepsilon}}-1)H_{2,j,\varepsilon}(x,\eta_{2,j,\varepsilon})\\
=&\frac{1}{\varepsilon^2}\int_{B_{\delta}(x_{j,\varepsilon)}} e^{u_{0,2}(x_{j,\varepsilon})+\beta_{j,\varepsilon}  +U_{2,j,\varepsilon}}\eta_{2,j,\varepsilon}+O\left(\frac{1}{\varepsilon^2} \int_{B_{\delta}(x_{j,\varepsilon)}} e^{u_{0,2}(x_{j,\varepsilon})+\beta_{j,\varepsilon}  +U_{2,j,\varepsilon}}|\eta_{2,j,\varepsilon}|^2 \right)\\
&+\frac{1}{\varepsilon^2} |D f_{2,j,\varepsilon}(x_{j,\varepsilon})|^2 \int_{B_{\delta}(x_{j,\varepsilon)}} e^{u_{0,2}(x_{j,\varepsilon})+\beta_{j,\varepsilon}  +U_{2,j,\varepsilon}}|x-x_{j,\varepsilon}|^2 \\
&+O\left(\frac{1}{\varepsilon^2}\int_{B_{\delta}(x_{j,\varepsilon)}} e^{u_{0,2}(x_{j,\varepsilon})+\beta_{j,\varepsilon}  +U_{2,j,\varepsilon}}|x-x_{j,\varepsilon}|^4\right)\\
&+ \frac{1}{\varepsilon^2}\int_{B_{\delta}(x_{j,\varepsilon)}} e^{u_{0,2}(x_{j,\varepsilon})+\beta_{j,\varepsilon}  +U_{2,j,\varepsilon}}  (e^{U_{2,j,\varepsilon}^*-U_{2,j,\varepsilon}}-1)H_{2,j,\varepsilon}(x,\eta_{2,j,\varepsilon})
\end{split}
\end{equation}
By Proposition \ref{pro42}, we have
\begin{equation}
\begin{split}
&\frac{1}{\varepsilon^2}\int_{B_{\delta}(x_{j,\varepsilon)}} e^{u_{0,2}(x_{j,\varepsilon})+\beta_{j,\varepsilon}  +U_{2,j,\varepsilon}}\eta_{2,j,\varepsilon}\\
= &\frac{e^{u_{0,2}(x_{j,\varepsilon})}}{2\pi}\int_{0}^{\mu_{j,\varepsilon}\delta} r e^{V_{2,j,\varepsilon}(r)} \int_{0}^{2\pi}\tilde{\eta}_{2,j,\varepsilon}d\theta dr\\
=&O(\mu_{j,\varepsilon}^{-2 }+e^{\beta_{j,\varepsilon}}).
\end{split}
\end{equation}
On the other hand, by Proposition \ref{pro41}, we have
\begin{equation}
 \frac{1}{\varepsilon^2}\int_{B_{\delta}(x_{j,\varepsilon)}} e^{u_{0,2}(x_{j,\varepsilon})+\beta_{j,\varepsilon}  +U_{2,j,\varepsilon}}|\eta_{2,j,\varepsilon}|^2=O(\mu_{j,\varepsilon}^{-2\tau}+e^{\tau'\beta_{j,\varepsilon}})^2.
\end{equation}
After a direct  calculation, we find
\begin{equation}
\begin{split}
&\frac{1}{\varepsilon^2} |D f_{2,j,\varepsilon}(x_{j,\varepsilon})|^2 \int_{B_{\delta}(x_{j,\varepsilon)}} e^{u_{0,2}(x_{j,\varepsilon})+\beta_{j,\varepsilon}  +U_{2,j,\varepsilon}}|x-x_{j,\varepsilon}|^2\\
=&O(\mu_{j,\varepsilon}^{-2+o(1)}\ln \mu_{j,\varepsilon} |D f_{2,j,\varepsilon}(x_{j,\varepsilon})|^2  )
\end{split}
\end{equation}
and
\begin{equation}\label{e0731}
\frac{1}{\varepsilon^2}\int_{B_{\delta}(x_{j,\varepsilon)}} e^{u_{0,2}(x_{j,\varepsilon})+\beta_{j,\varepsilon}  +U_{2,j,\varepsilon}}|x-x_{j,\varepsilon}|^4= O(\delta^{2+o(1)}\mu_{j,\varepsilon}^{-2+o(1)}).
\end{equation}
Combining \eqref{e061}-\eqref{e0731}, we conclude that
\begin{equation}
\tilde{M}_{1,j,\varepsilon}-m_{1,j,\varepsilon}=O(|D f_{2,j,\varepsilon}(x_{j,\varepsilon})|^2\mu_{j,\varepsilon}^{-2+o(1)} \ln \mu_{j,\varepsilon})+O(\mu_{j,\varepsilon}^{-2+o(1)}+e^{\beta_{j,\varepsilon}}).
\end{equation}
Recall that $(\tilde{M}_{1,j,\varepsilon},\tilde{M}_{2,j,\varepsilon})$ satisfies
\begin{equation}
\tilde{M}_{1,j,\varepsilon}+\tilde{M}_{2,j,\varepsilon}=\frac{   \tilde{M}_{1,j,\varepsilon}\tilde{M}_{2,j,\varepsilon}     }{4\pi},
\end{equation}
and
 $\sum_{j=1}^k M_{i,j,\varepsilon}=8k\pi$, $i=1,2.$
By these and \eqref{e104},  we obtain
\begin{equation}\begin{split}
&\sum_{i=1,2}\sum_{j=1,\cdots,k}(\tilde{M}_{i,j,\varepsilon}-M_{i,j,\varepsilon})
=  \sum_{i=1,2}\sum_{j=1,\cdots,k}(\tilde{M}_{i,j,\varepsilon}-8\pi)\\
= &\frac{1}{2}\sum_{i=1,2}\sum_{j=1,\cdots,k}\frac{   (\tilde{M}_{i,j,\varepsilon}-8\pi)^2 }{\tilde{M}_{i,j,\varepsilon}-4\pi}
\end{split}
\end{equation}
We thus conclude that  $$\tilde{M}_{i,j,\varepsilon}=8\pi+O( |\tilde{M}_{i,j,\varepsilon}-{m}_{i,j,\varepsilon}|^{\frac{1}{2}} ).$$ 
Applying this  to the equations \eqref{e061}-\eqref{e0731}, the term $o(1)$ in the exponent of $\mu_{j,\varepsilon}$ can be removed.  We thus obtain \eqref{e053}.

\end{proof}
By Propositions \ref{pro41} and \ref{pro42} and Lemma \ref{lem42}, we can improve Lemma \ref{lem26} as follows.
\begin{lem}\label{lem44} For $i=1,2,$ it holds
\begin{equation}\label{re3}
u_{i,\varepsilon}=\frac{1}{|\Omega|}\int_{\Omega}u_{i,\varepsilon}+\sum_{j=1}^k \tilde{M}_{i,j,\varepsilon} G(x_{j,\varepsilon},x)+O(|\tilde{M}_{i,j,\varepsilon}-m_{i,j,\varepsilon}|)
\end{equation}
 {   in   } $C^{1}(\Omega\setminus \cup_{j=1}^k B_{\delta}(x_{j,\varepsilon})).$
\end{lem}
%
%
In the next lemma, we show that the heights of the bubbles are comparable.

\begin{lem}\label{lemma53}
It holds
\begin{equation}\label{e078}
\begin{split}
-\beta_{j,\varepsilon}&=8\pi\gamma(x_{j,\varepsilon},x_{j,\varepsilon})+8\pi \sum_{l\not=j} G(x_{j,\varepsilon},x_{l,\varepsilon})
-I_{1,j,\varepsilon}-4\ln \varepsilon+\frac{1}{|\Omega|}\int_{\Omega}{u_{1,\varepsilon}}+o(1)\\  
&= 8\pi\gamma(x_{j,\varepsilon},x_{j,\varepsilon})+8\pi \sum_{l\not=j} G(x_{j,\varepsilon},x_{l,\varepsilon})
-I_{2,j,\varepsilon}-4\ln \varepsilon+\frac{1}{|\Omega|}\int_{\Omega}{u_{2,\varepsilon}}+o(1)  
\end{split}
\end{equation}

\end{lem}

\begin{proof}
It follows from Proposition \ref{pro41} that
\begin{equation}\label{e079}
\begin{split}&u_{1,\varepsilon}(x)\\
=& \tilde{M}_{1,j,\varepsilon} (\gamma(x,x_{j,\varepsilon})-\gamma(x_{j,\varepsilon},x_{j,\varepsilon})  )
 +\sum_{l\not= j} \tilde{M}_{1,l,\varepsilon} (G(x,x_{l,\varepsilon})-G(x_{j,\varepsilon},x_{l,\varepsilon}))\\
 & +\beta_{j,\varepsilon}+\frac{\tilde{M}_{1,j,\varepsilon}}{2\pi}\ln\mu_{j,\varepsilon}\delta+I_{1,j,\varepsilon}+O(\mu_{j,\varepsilon}^{-2\tau}+e^{\tau'\beta_{j,\varepsilon}}) +O(|\tilde{M}_{i,j,\varepsilon}-m_{i,j,\varepsilon}|),
 \end{split}
\end{equation}
for $x\in \partial B_{\delta}(x_{j,\varepsilon})$.
By this and Lemma \ref{lem44}, we find
\begin{equation*}
\begin{split}
(1-\frac{\tilde{M}_{1,j,\varepsilon}}{4\pi})\beta_{j,\varepsilon}
&=\tilde{M}_{1,j,\varepsilon}  \gamma(x_{j,\varepsilon},x_{j,\varepsilon})    +\sum_{l\not= j}\tilde{M}_{1,l,\varepsilon}  G(x_{j,\varepsilon},x_{l,\varepsilon})-\frac{\tilde{M}_{1,j,\varepsilon}}{2\pi}\ln\varepsilon+\frac{1}{|\Omega|}\int_{\Omega}{u_{1,\varepsilon}}\\
&\quad +O(\mu_{j,\varepsilon}^{-2\tau}+e^{\tau'\beta_{j,\varepsilon}})+ O(|\tilde{M}_{i,j,\varepsilon}-m_{i,j,\varepsilon}|) .
 \end{split}
\end{equation*}
So, \eqref{e078} follows from Lemma \ref{lem42}.

\end{proof}

\begin{rem}\label{rem51}
In view of \eqref{e079}, we know that 
\begin{equation}\label{re6}
I_{1,j,\varepsilon}-I_{2,j,\varepsilon}=\frac{1}{|\Omega|}\int_{\Omega} (u_{1,\varepsilon}-u_{2,\varepsilon})+o(1),\quad j=1,\cdots, k.
\end{equation}
Furthermore, by the classification of $(V_{1,j,\varepsilon},V_{2,j,\varepsilon})$ and Lemma \ref{lem42}, we obtain
\begin{equation}\label{re7}
u_{0,1}(x_{j,\varepsilon})-u_{0,2}(x_{j,\varepsilon})=I_{2,j,\varepsilon}-I_{1,j,\varepsilon}+o(1),\quad j=1,\cdots, k. 
\end{equation}
Thus, by \eqref{re6} and \eqref{re7}, (1) of Theorem \ref{main1} is proved.

\end{rem}

With this lemma, we can obtain a estimate for $\eta_{i,j,\varepsilon}$ in $\Omega\setminus \cup_{j=1}^k B_{\delta}(x_{j,\varepsilon})$.

\begin{prop}
For  $x\in \Omega\setminus \cup_{j=1}^k B_{\delta}(x_{j,\varepsilon})$, it holds
\begin{equation}
|\eta_{i,j,\varepsilon}(x)|= O( \mu_{j,\varepsilon}^{-2\tau}+e^{\tau'\beta_{j,\varepsilon}} +|\tilde{M}_{i,j,\varepsilon}-m_{i,j,\varepsilon}|),\,\,i=1,2.
\end{equation}
\end{prop}

\begin{proof}
By \eqref{e079}  
\begin{equation}
\begin{split}
\eta_{i,j,\varepsilon} =& u_{i,\varepsilon}(x)-\beta_{j,\varepsilon}-\sum_{l=1}^k \tilde{M}_{i,l,\varepsilon}  G(x,x_{l,\varepsilon}) -\frac{\tilde{M}_{i,j,\varepsilon}}{2\pi}\ln\mu_{j,\varepsilon}-I_{i,j,\varepsilon}+\tilde{M}_{i,j,\varepsilon} \gamma(x_{j,\varepsilon},x_{j,\varepsilon})\\
&+\sum_{l\not=j} \tilde{M}_{i,l,\varepsilon} G(x_{j,\varepsilon},x_{l,\varepsilon})+O(\mu_{j,\varepsilon}^{-2\tau} +e^{\tau' \beta_{j,\varepsilon}})+O(|\tilde{M}_{i,j,\varepsilon}-m_{i,j,\varepsilon}|)\\
=& u_{i,\varepsilon}-\frac{1}{|\Omega|}\int_{\Omega} u_{i,\varepsilon}-\sum_{l=1}^k \tilde{M}_{i,l,\varepsilon}  G(x,x_{l,\varepsilon})+O(\mu_{j,\varepsilon}^{-2\tau} +e^{\tau' \beta_{j,\varepsilon}})+O(|\tilde{M}_{i,j,\varepsilon}-m_{i,j,\varepsilon}|)\\
=&O(\mu_{j,\varepsilon}^{-2\tau} +e^{\tau' \beta_{j,\varepsilon}}+|\tilde{M}_{i,j,\varepsilon}-m_{i,j,\varepsilon}|),
 \end{split}
\end{equation}
where the last equality follows from Lemma \ref{lem44}.

\end{proof}

\begin{lem}\label{lem47}
It holds $$D f_{1,j}(q_j)=D f_{2,j}(q_j)=0,\quad j=1,\cdots,k.$$
\end{lem}

\begin{proof}
Since $(u_{1,\varepsilon},u_{2,\varepsilon})$ is doubly periodic on $\partial \Omega$, by integrating \eqref{e002}, we have
\begin{equation}\label{e009}
\begin{cases}
8k\pi=\frac{1}{\varepsilon^2}\int_{\Omega}e^{u_{2,\varepsilon}+u_{0,2}}(1-e^{u_{1,\varepsilon}+u_{0,1}})=\sum_{j=1}^k \frac{1}{\varepsilon^2}\int_{\Omega_j}e^{u_{2,\varepsilon}+u_{0,2}}(1-e^{u_{1,\varepsilon}+u_{0,1}}),\\
8k\pi=\frac{1}{\varepsilon^2}\int_{\Omega}e^{u_{1,\varepsilon}+u_{0,1}}(1-e^{u_{2,\varepsilon}+u_{0,2}})=\sum_{j=1}^k \frac{1}{\varepsilon^2}\int_{\Omega_j}e^{u_{1,\varepsilon}+u_{0,1}}(1-e^{u_{2,\varepsilon}+u_{0,2}}).
\end{cases}
\end{equation}
Recall that
\begin{equation}\label{e010}
\begin{cases}
\tilde{M}_{1,j,\varepsilon}=\frac{1}{\varepsilon^2}\int_{ \Omega_j} e^{u_{0,2}(x_{j,\varepsilon})+\beta_{j,\varepsilon}+U^*_{2,j,\varepsilon}}+\frac{1}{\varepsilon^2}\int_{\mathbb{R}^2\setminus\Omega_j} e^{u_{0,2}(x_{j,\varepsilon})+\beta_{j,\varepsilon}+U^*_{2,j,\varepsilon}},\\
\tilde{M}_{2,j,\varepsilon}=\frac{1}{\varepsilon^2}\int_{ \Omega_j} e^{u_{0,1}(x_{j,\varepsilon})+\beta_{j,\varepsilon}+U^*_{1,j,\varepsilon}}+\frac{1}{\varepsilon^2}\int_{\mathbb{R}^2\setminus\Omega_j} e^{u_{0,1}(x_{j,\varepsilon})+\beta_{j,\varepsilon}+U^*_{1,j,\varepsilon}}.
\end{cases}
\end{equation}
By \eqref{e060}, we have
\begin{equation}\label{e063}
\begin{cases}
\Delta \eta_{1,j,\varepsilon}=\Delta u_{1, \varepsilon}-\Delta U_{1,j,\varepsilon}^* -\sum_{l=1}^k  {M}_{1,l,\varepsilon}\\
\Delta \eta_{2,j,\varepsilon}=\Delta u_{2, \varepsilon}-\Delta U_{2,j,\varepsilon}^* -\sum_{l=1}^k  {M}_{2,l,\varepsilon}
\end{cases}\quad x\in \Omega_j.
\end{equation}
By integrating \eqref{e063} over $\Omega_j$, we have
\begin{equation*}
\begin{cases}
\int_{\partial \Omega_j}\frac{\partial \eta_{1,j,\varepsilon}}{\partial\nu}=\tilde{M}_{1,j,\varepsilon}-\frac{1}{\varepsilon^2}\int_{\mathbb{R}^2\setminus\Omega_j} e^{u_{0,2}(x_{j,\varepsilon})+\beta_{j,\varepsilon}+U_{2,j,\varepsilon}^*}-M_{1,j,\varepsilon},  \\
\int_{\partial \Omega_j}\frac{\partial \eta_{2,j,\varepsilon}}{\partial\nu}=\tilde{M}_{2,j,\varepsilon}-\frac{1}{\varepsilon^2}\int_{\mathbb{R}^2\setminus\Omega_j} e^{u_{0,1}(x_{j,\varepsilon})+\beta_{j,\varepsilon}+U_{1,j,\varepsilon}^*}-M_{2,j,\varepsilon},
\end{cases}
\end{equation*}
and
\begin{equation}\label{e022}
\begin{cases}
\sum_{j=1}^k\left(\int_{\partial \Omega_j}\frac{\partial \eta_{1,j,\varepsilon}}{\partial\nu}+ \frac{1}{\varepsilon^2}\int_{\mathbb{R}^2\setminus\Omega_j} e^{u_{0,2}(x_{j,\varepsilon})+\beta_{j,\varepsilon}+U_{2,j,\varepsilon}^*}\right)=\sum_{j=1}^{k}( \tilde{M}_{1,j,\varepsilon} -M_{1,j,\varepsilon}  ),\\
 \sum_{j=1}^k\left(\int_{\partial \Omega_j}\frac{\partial \eta_{2,j,\varepsilon}}{\partial\nu}+ \frac{1}{\varepsilon^2}\int_{\mathbb{R}^2\setminus\Omega_j} e^{u_{0,1}(x_{j,\varepsilon})+\beta_{j,\varepsilon}+U^*_{1,j,\varepsilon}}\right)=\sum_{j=1}^{k}( \tilde{M}_{2,j,\varepsilon} -M_{2,j,\varepsilon}  ).
\end{cases}
\end{equation}
Set
\begin{align*}
\psi^*_{1,j}=\mu_{j,\varepsilon}(x-x_{j,\varepsilon}^*){V_{1,j,\varepsilon}}'(\mu_{j,\varepsilon}(x-x_{j,\varepsilon}^*))+2,\\
\psi^*_{2,j}=\mu_{j,\varepsilon}(x-x_{j,\varepsilon}^*){V_{2,j,\varepsilon}}'(\mu_{j,\varepsilon}(x-x_{j,\varepsilon}^*))+2 ,
\end{align*} then direct computation shows
\begin{equation}\label{e011}
\begin{cases}
\Delta \psi^*_{1,j}+\frac{1}{\varepsilon^2}e^{u_{0,2}(x_{j,\varepsilon})+\beta_{j,\varepsilon}+U^*_{2,j,\varepsilon}}\psi^*_{2,j}=0\\
\Delta \psi^*_{2,j}+\frac{1}{\varepsilon^2}e^{u_{0,1}(x_{j,\varepsilon})+\beta_{j,\varepsilon}+U^*_{1,j,\varepsilon}}\psi^*_{1,j}=0
\end{cases}
\end{equation}
Similarly, we set 
\begin{align*}
\psi_{1,j}=\mu_{j,\varepsilon}(x-x_{j,\varepsilon} ){V_{1,j,\varepsilon}}'(\mu_{j,\varepsilon}(x-x_{j,\varepsilon} ))+2,\\
\psi_{2,j}=\mu_{j,\varepsilon}(x-x_{j,\varepsilon} ){V_{2,j,\varepsilon}}'(\mu_{j,\varepsilon}(x-x_{j,\varepsilon} ))+2.
\end{align*}
By Green's identity  and \eqref{e011}, we obtain
\begin{equation}\label{e012}
\begin{split}
&\sum_{j=1}^k\int_{\partial \Omega_j}\psi^*_{2,j}\frac{\partial \eta_{1,j,\varepsilon}}{\partial \nu}+\psi^*_{1,j}\frac{\partial \eta_{2,j,\varepsilon}}{\partial \nu}-\eta_{1,j,\varepsilon}\frac{\partial \psi^*_{2,j}}{\partial \nu}-\eta_{2,j,\varepsilon}\frac{\partial \psi^*_{1,j}}{\partial \nu}
\\
=&\sum_{j=1}^k\int_{\Omega_j} \psi^*_{2,j}\Delta\eta_{1,j,\varepsilon}-\eta_{1,j,\varepsilon}\Delta\psi^*_{2,j}+\psi^*_{1,j}\Delta\eta_{2,j,\varepsilon}-\eta_{2,j,\varepsilon}\Delta\psi^*_{1,j}\\
=&\sum_{j=1}^k \int_{\Omega_j}\Big(-\frac{1}{\varepsilon^2}e^{u_{0,2}(x_{j,\varepsilon})+\beta_{\varepsilon}+U^*_{2,j,\varepsilon}}
(H_{2,j,\varepsilon}(x,\eta_{2,j,\varepsilon}) -\eta_{2,j,\varepsilon})\psi^*_{2,j}\\
&\quad -\frac{1}{\varepsilon^2}e^{u_{0,1}(x_{j,\varepsilon})+\beta_{j,\varepsilon}+U^*_{1,j,\varepsilon}}
(H_{1,j,\varepsilon}(x,\eta_{1,j,\varepsilon}) -\eta_{1,j,\varepsilon})\psi^*_{1,j} +\frac{1}{\varepsilon^2}e^{\sum_{i=1}^2(u_{i,\varepsilon}+u_{0,i})}(\psi^*_{1,j}+\psi^*_{2,j})\Big).
\end{split}
\end{equation}

We estimate the right hand side of \eqref{e012}. As in \eqref{e062}, we have
\begin{equation}\label{e013}
\begin{split}
&\frac{1}{\varepsilon^2}\int_{B_{\delta}(x_{j,\varepsilon})} e^{u_{0,2}(x_{j,\varepsilon})+\beta_{\varepsilon}+U^*_{2,j,\varepsilon}} H_{2,j,\varepsilon}(x,0)\psi^*_{2,j}\\
=& \frac{1}{\varepsilon^2}\int_{B_{\delta}(x_{j,\varepsilon})} e^{u_{0,2}(x_{j,\varepsilon})+\beta_{\varepsilon}+U_{2,j,\varepsilon}} H_{2,j,\varepsilon}(x,0)\psi_{2,j}\\
& + \frac{1}{\varepsilon^2}\int_{B_{\delta}(x_{j,\varepsilon})} e^{u_{0,2}(x_{j,\varepsilon})+\beta_{\varepsilon}+U_{2,j,\varepsilon}} (e^{U^*_{2,j,\varepsilon}-U_{2,j,\varepsilon}   }-1)   H_{2,j,\varepsilon}(x,0)(\psi^*_{2,j}-\psi_{2,j})\\
&+  \frac{1}{\varepsilon^2}\int_{B_{\delta}(x_{j,\varepsilon})} e^{u_{0,2}(x_{j,\varepsilon})+\beta_{\varepsilon}+U_{2,j,\varepsilon}}     H_{2,j,\varepsilon}(x,0)(\psi^*_{2,j}-\psi_{2,j})\\
&+\frac{1}{\varepsilon^2}\int_{B_{\delta}(x_{j,\varepsilon})} e^{u_{0,2}(x_{j,\varepsilon})+\beta_{\varepsilon}+U_{2,j,\varepsilon}}  (e^{U^*_{2,j,\varepsilon}-U_{2,j,\varepsilon}   }-1)    H_{2,j,\varepsilon}(x,0) \psi_{2,j} \\
=&-C_{2,j,\varepsilon} |D f_{2,j,\varepsilon}(x_{j,\varepsilon})|^2\mu_{j,\varepsilon}^{-2}\ln\mu_{j,\varepsilon}+O(\mu_{j,\varepsilon}^{-2})  
\end{split}
\end{equation}
for some positive constant $C_{2,j,\varepsilon}$ and $\lim_{\varepsilon\to 0} C_{2,j,\varepsilon}>0$.

By \eqref{e014} and Proposition \ref{pro41}, we obtain
\begin{equation}\label{e016}
\begin{split}
&\frac{1}{\varepsilon^2}\int_{B_{\delta}(x_{j,\varepsilon})} e^{u_{0,2}(x_{j,\varepsilon})+\beta_{j,\varepsilon}+U^*_{2,j,\varepsilon}}(H_{2,j,\varepsilon}(x,\eta_{2,j,\varepsilon})-\eta_{2,j,\varepsilon})\psi^*_{2,j,\varepsilon}\\
=&-C_{2,j,\varepsilon} |D f_{2,j,\varepsilon}(x_{j,\varepsilon})|^2\mu_{j,\varepsilon}^{-2}\ln\mu_{j,\varepsilon}+ O(\mu_{j,\varepsilon}^{-2})+O((\mu^{-2\tau}+e^{\tau'\beta_{j,\varepsilon}})^2).
\end{split}
\end{equation}
Similarly,
\begin{equation}\label{e017}
\begin{split}
&\frac{1}{\varepsilon^2}\int_{B_{\delta}(x_{j,\varepsilon})} e^{u_{0,1}(x_{j,\varepsilon})+\beta_{\varepsilon}+U^*_{1,j,\varepsilon}}(H_{1,j,\varepsilon}(x,\eta_{1,j,\varepsilon})-\eta_{1,j,\varepsilon})\psi_{1,j,\varepsilon}\\
=&-C_{1,j,\varepsilon}|D f_{1,j,\varepsilon}(x_{j,\varepsilon})|^2\mu_{j,\varepsilon}^{-2}\ln\mu_{j,\varepsilon}+O(\mu_{j,\varepsilon}^{-2})+O((\mu^{-2\tau}+e^{\tau'\beta_{j,\varepsilon}})^2),
\end{split}
\end{equation}
where $\lim_{\varepsilon\to 0} C_{1,j,\varepsilon}>0$.
As the calculation of \eqref{e013}, we have
\begin{equation}\label{e018}
\frac{1}{\varepsilon^2}\int_{  \Omega_j } e^{\sum_{i=1}^2(u_{i,\varepsilon}+u_{0,i})}(\psi^*_{1,j} +\psi^*_{2,j}  )=B_{j,\varepsilon}e^{\beta_{j,\varepsilon}}
\end{equation}
for some positive constant $B_{j,\varepsilon}$ with $\lim_{\varepsilon\to 0} B_{j,\varepsilon}>0$..
Thus, we conclude that
\begin{equation}\label{e019}
\begin{split}
&  \sum_{j=1}^k\int_{\Omega_j} \psi^*_{2,j}\Delta\eta_{1,j,\varepsilon}-\eta_{1,j,\varepsilon}\Delta\psi^*_{2,j}+\psi^*_{1,j}\Delta\eta_{2,j,\varepsilon}-\eta_{2,j,\varepsilon}\Delta\psi^*_{1,j}\\
=&\sum_{j=1}^k \left(B_{j,\varepsilon} e^{\beta_{j,\varepsilon}}+C_{1,j,\varepsilon}|D f_{1,j,\varepsilon}(x_{j,\varepsilon})|^2\mu_{j,\varepsilon}^{-2}\ln\mu_{j,\varepsilon} + C_{2,j,\varepsilon} |D f_{2,j,\varepsilon}(x_{j,\varepsilon})|^2\mu_{j,\varepsilon}^{-2}\ln\mu_{j,\varepsilon}\right)\\
&+O(\mu_{j,\varepsilon}^{-2})+O((\mu^{-2\tau}+e^{\tau\beta_{j,\varepsilon}})^2)
\end{split}
\end{equation}

On the other hand, by \eqref{e022} and Lemma \ref{lem42},
\begin{equation}\label{e020}
\begin{split}
&\sum_{j=1}^k\int_{\partial \Omega_j}\psi^*_{2,j}\frac{\partial \eta_{1,j,\varepsilon}}{\partial \nu}+\psi^*_{1,j}\frac{\partial \eta_{2,j,\varepsilon}}{\partial \nu}-\eta_{1,j,\varepsilon}\frac{\partial \psi^*_{2,j}}{\partial \nu}-\eta_{2,j,\varepsilon}\frac{\partial \psi^*_{1,j}}{\partial \nu}
\\
=& \sum_{j=1}^k \left( (-\frac{\tilde{M}_{2,j,\varepsilon}}{2\pi}+2)\int_{\partial \Omega_j}\frac{\partial \eta_{1,j,\varepsilon}}{\partial \nu}+ (-\frac{\tilde{M}_{1,j,\varepsilon}}{2\pi}+2)\int_{\partial \Omega_j}\frac{\partial \eta_{2,j,\varepsilon}}{\partial \nu}\right)+O(\mu_{j,\varepsilon}^{-2}(\mu_{j,\varepsilon}^{-2\tau}+e^{\tau'\beta_{\varepsilon}}  ))\\
=&-  \sum_{i=1,2}\sum_{j=1,\cdots,k}\frac{(\tilde{M}_{i,j,\varepsilon}-8\pi)^2}{\tilde{M}_{i,j,\varepsilon}-4\pi} +
O(|\tilde{M}_{i,j,\varepsilon}-8\pi||\tilde{M}_{i,j,\varepsilon}-m_{i,j,\varepsilon}|)+O(\mu_{j,\varepsilon}^{-2}(\mu_{j,\varepsilon}^{-2\tau}+e^{\tau'\beta_{\varepsilon}}  ))
 \\
=&- \sum_{i=1,2}\sum_{j=1,\cdots,k}\frac{(\tilde{M}_{i,j,\varepsilon}-8\pi)^2}{\tilde{M}_{i,j,\varepsilon}-4\pi}+ O( |\tilde{M}_{i,j,\varepsilon}-m_{i,j,\varepsilon}|^{3/2})+O(\mu_{j,\varepsilon}^{-2}(\mu_{j,\varepsilon}^{-2\tau}+e^{\tau'\beta_{\varepsilon}}  )).
\end{split}
\end{equation}

Combing \eqref{e019} and \eqref{e020}, we obtain
\begin{equation}\label{e025}
\begin{split}
&\sum_{j=1}^k (B_{j,\varepsilon} e^{\beta_{j,\varepsilon}}+\sum_{i=1,2}C_{i,j,\varepsilon}|D f_{i,j,\varepsilon}(x_{j,\varepsilon})|^2\mu_{j,\varepsilon}^{-2}\ln\mu_{j,\varepsilon} ) +\sum_{i=1,2}\sum_{j=1,\cdots,k}\frac{(\tilde{M}_{i,j,\varepsilon}-8\pi)^2}{\tilde{M}_{i,j,\varepsilon}-4\pi}\\
=&  o(e^{\beta_{j,\varepsilon}})+O(\mu_{j,\varepsilon}^{-2}).
\end{split}
\end{equation}
Since  $B_{j,\varepsilon}$, $C_{1,j,\varepsilon}$ and $C_{2,j,\varepsilon}$ are positive, we obtain
\begin{equation}
|D f_{1,j,\varepsilon}(x_{j,\varepsilon})|^2=o(1)\mbox{  and  }|D f_{2,j,\varepsilon}(x_{j,\varepsilon})|=o(1).
\end{equation}
It follows that $$D f_{1,j}(q_j)=D f_{2,j}(q_j)=0,\,\,j=1,\cdots,k .$$

\end{proof}

\begin{rem}By Lemma \ref{lem42}   and  Lemma \ref{lem47}, we obtain, for $i=1,2,$ $j=1,\cdots,k$,
\begin{equation}\label{e080}
|D f_{i,j,\varepsilon}(x_{j,\varepsilon})|= O(\sum_{j=1,\cdots,k}|8\pi-M_{i,j,\varepsilon}|)
+O(|q_j-x_{j,\varepsilon}|)
=O(\mu_{j,\varepsilon}^{-1} +e^{\beta_{j,\varepsilon}}).
\end{equation}
Furthermore, applying  \eqref{e080} to \eqref{e112}, \eqref{re4} and \eqref{e113}, we can refine Proposition \ref{pro31} and show that
\begin{equation}\label{e081}
N_{\varepsilon}\leq C  \quad\text{  for  } 0<\tau<1\quad 0<\tau'<1.
\end{equation}
For convenience,   we set $\tau'=\tau$.
\end{rem}

With \eqref{e080} and \eqref{e081}, we can refine the estimate on \eqref{e012} and obtain the following estimate.

\begin{lem}\label{final-lem} It holds
\begin{equation*}
\mathcal{D}^{(2)}({\bf  q})\leq 0
\end{equation*}
\end{lem}

\begin{proof}

Let $\theta_{j,\varepsilon}$ be a small parameter.  By    \eqref{e080} and \eqref{e081},
similar to \eqref{e013} and \eqref{e016}, we find

\begin{equation}\label{e083}
\begin{split}
&\frac{1}{\varepsilon^2}\int_{B_{\theta_{j,\varepsilon}}(x_{j,\varepsilon})} e^{u_{0,2}(x_{j,\varepsilon})+\beta_{j,\varepsilon}+U^*_{2,j,\varepsilon}}(H_{2,j,\varepsilon}(x,\eta_{2,j,\varepsilon})-\eta_{2,j,\varepsilon})\psi^*_{2,j,\varepsilon}\\
=&O\Big( \frac{1}{\varepsilon^2} |Df_{2,j,\varepsilon}(x_j,\varepsilon)|^2\int_{B_{\theta_{j,\varepsilon}}(x_{j,\varepsilon})} e^{u_{0,2}(x_{j,\varepsilon})+\beta_{j,\varepsilon}+U^*_{2,j,\varepsilon}}|x-x_{j,\varepsilon}|^2\\
&\quad+ \frac{1}{\varepsilon^2} \int_{B_{\theta_{j,\varepsilon}}(x_{j,\varepsilon})} e^{u_{0,2}(x_{j,\varepsilon})+\beta_{j,\varepsilon}+U^*_{2,j,\varepsilon}}|x-x_{j,\varepsilon}|^4\\
& \quad+ \frac{1}{\varepsilon^2} |Df_{2,j,\varepsilon}(x_j,\varepsilon)|\int_{B_{\theta_{j,\varepsilon}}(x_{j,\varepsilon})} e^{u_{0,2}(x_{j,\varepsilon})+\beta_{j,\varepsilon}+U^*_{2,j,\varepsilon}}|x-x_{j,\varepsilon}| |\eta_{2,j,\varepsilon}|\\
& \quad+ \frac{1}{\varepsilon^2} \int_{B_{\theta_{j,\varepsilon}}(x_{j,\varepsilon})} e^{u_{0,2}(x_{j,\varepsilon})+\beta_{j,\varepsilon}+U^*_{2,j,\varepsilon}} |\eta_{2,j,\varepsilon}|^2\Big)\\
=&O\Big( (\mu_{j,\varepsilon}^{-2}+e^{\beta_{j,\varepsilon}})\mu_{j,\varepsilon}^{-2}\ln \mu_{j,\varepsilon}\theta_{j,\varepsilon}+\theta_{j,\varepsilon}^2\mu_{j,\varepsilon}^{-2} +\mu_{j,\varepsilon}^{-4\tau} +e^{2\tau\beta_{j,\varepsilon}}             \Big)
\end{split}
\end{equation}
By this and \eqref{e058}, we find

\begin{equation}\label{e085}
\begin{split}
& \frac{1}{\varepsilon^2} \int_{ \Omega_j}e^{u_{0,2}(x_{j,\varepsilon}) +\beta_{j,\varepsilon}+U_{2,j,\varepsilon}^*  } ( H_{2,j,\varepsilon}(x,\eta_{2,j,\varepsilon})-\eta_{2,j,\varepsilon})\psi^*_{2,j,\varepsilon}\\
=&  \frac{1}{\varepsilon^2} \int_{ \Omega_j\setminus B_{\theta_{j,\varepsilon}}(x_{j,\varepsilon})}e^{u_{0,2}(x_{j,\varepsilon}) +\beta_{j,\varepsilon}+U_{2,j,\varepsilon}^*  } ( H_{2,j,\varepsilon}(x,\eta_{2,j,\varepsilon})-\eta_{2,j,\varepsilon})\psi^*_{2,j,\varepsilon}\\
&  + O\Big( (\mu_{j,\varepsilon}^{-2}+e^{\beta_{j,\varepsilon}})\mu_{j,\varepsilon}^{-2}\ln \mu_{j,\varepsilon}\theta_{j,\varepsilon}+\theta_{j,\varepsilon}^2\mu_{j,\varepsilon}^{-2} +\mu_{j,\varepsilon}^{-4\tau} +e^{2\tau\beta_{j,\varepsilon}}             \Big)\\
=&(-\frac{\tilde{M}_{2,j,\varepsilon}}{2\pi}+2)\mu_{j,\varepsilon}^{2-\frac{\tilde{M}_{2,j,\varepsilon}}{2\pi}    }  \int_{\Omega_j\setminus B_{\theta_{j,\varepsilon}}(x_{j,\varepsilon})} \frac{(e^{f_{2,j,\varepsilon}}-1) e^{I_{2,j,\varepsilon}+u_{0,2}(x_{j,\varepsilon})}}{|x-x_{j,\varepsilon}|^{\frac{\tilde{M}_{2,j,\varepsilon}}{2\pi}}}\\
&  + O\Big( (\mu_{j,\varepsilon}^{-2}+e^{\beta_{j,\varepsilon}})\mu_{j,\varepsilon}^{-2}\ln \mu_{j,\varepsilon}\theta_{j,\varepsilon}+\theta_{j,\varepsilon}^2\mu_{j,\varepsilon}^{-2} +\mu_{j,\varepsilon}^{-4\tau} +e^{2\tau\beta_{j,\varepsilon}}             \Big).
\end{split}
\end{equation}
Similarly,

\begin{equation}\label{e086}
\begin{split}
& \frac{1}{\varepsilon^2} \int_{ \Omega_j}e^{u_{0,1}(x_{j,\varepsilon}) +\beta_{j,\varepsilon}+U_{1,j,\varepsilon}^*  } ( H_{1,j,\varepsilon}(x,\eta_{1,j,\varepsilon})-\eta_{1,j,\varepsilon})\psi^*_{1,j,\varepsilon}\\
=&(-\frac{\tilde{M}_{1,j,\varepsilon}}{2\pi}+2)\mu_{j,\varepsilon}^{2-\frac{\tilde{M}_{1,j,\varepsilon}}{2\pi}    }  \int_{\Omega_j\setminus B_{\theta_{j,\varepsilon}}(x_{j,\varepsilon})} \frac{(e^{f_{1,j,\varepsilon}}-1) e^{I_{1,j,\varepsilon}+u_{0,1}(x_{j,\varepsilon})}}{|x-x_{j,\varepsilon}|^{\frac{\tilde{M}_{1,j,\varepsilon}}{2\pi}}}\\
& + O\Big( (\mu_{j,\varepsilon}^{-2}+e^{\beta_{j,\varepsilon}})\mu_{j,\varepsilon}^{-2}\ln \mu_{j,\varepsilon}\theta_{j,\varepsilon}+\theta_{j,\varepsilon}^2\mu_{j,\varepsilon}^{-2} +\mu_{j,\varepsilon}^{-4\tau} +e^{2\tau\beta_{j,\varepsilon}}             \Big) .
\end{split}
\end{equation}
On the other hand,

\begin{equation}\label{e087}
\begin{split}
&\sum_{j=1}^k\int_{\partial \Omega_j}\psi^*_{2,j}\frac{\partial \eta_{1,j,\varepsilon}}{\partial \nu}+\psi^*_{1,j}\frac{\partial \eta_{2,j,\varepsilon}}{\partial \nu}-\eta_{1,j,\varepsilon}\frac{\partial \psi^*_{2,j}}{\partial \nu}-\eta_{2,j,\varepsilon}\frac{\partial \psi^*_{1,j}}{\partial \nu}\\
  =& \sum_{j=1}^k \left( (-\frac{M_{1,j,\varepsilon}}{2\pi}+2)\int_{\partial \Omega_j}\frac{\partial \eta_{2,j,\varepsilon}}{\partial \nu}+ (-\frac{M_{2,j,\varepsilon}}{2\pi}+2)\int_{\partial \Omega_j}\frac{\partial \eta_{1,j,\varepsilon}}{\partial \nu}\right)+O(\mu_{j,\varepsilon}^{-2}(\mu_{j,\varepsilon}^{-2\tau}+e^{\tau\beta_{\varepsilon}}  ))\\
=&  \sum_{j=1}^k  \left( (\frac{\tilde{M}_{1,j,\varepsilon}}{2\pi}-2)\mu_{j,\varepsilon}^{ -\frac{\tilde{M}_{1,j,\varepsilon}}{2\pi}+2  }
\int_{\mathbb{R}^2\setminus \Omega_j} \frac{e^{u_{0,1}(x_{j,\varepsilon})+I_{1,j,\varepsilon}}}{|x-x_{j,\varepsilon}|^{\frac{\tilde{M}_{1,j,\varepsilon}}{2\pi}}}\right.\\
&\qquad \left. +(\frac{\tilde{M}_{2,j,\varepsilon}}{2\pi}-2)\mu_{j,\varepsilon}^{ -\frac{\tilde{M}_{2,j,\varepsilon}}{2\pi}+2  }
\int_{\mathbb{R}^2\setminus \Omega_j} \frac{e^{u_{0,2}(x_{j,\varepsilon})+I_{2,j,\varepsilon}}}{|x-x_{j,\varepsilon}|^{\frac{\tilde{M}_{2,j,\varepsilon}}{2\pi}}}\right)- \sum_{i=1,2}\sum_{j=1,\cdots,k}\frac{(\tilde{M}_{i,j,\varepsilon}-8\pi)^2}{\tilde{M}_{i,j,\varepsilon}-4\pi}\\
& +O(\mu_{j,\varepsilon}^{-2}(\mu_{j,\varepsilon}^{-2\tau}+e^{\tau\beta_{\varepsilon}}  )) +O(  |\tilde{M}_{i,j,\varepsilon}-m_{i,j,\varepsilon}|^{3/2}).
\end{split}
\end{equation}

Combining \eqref{e085}, \eqref{e086} and \eqref{e087}, we have
\begin{equation}\label{e088}
\begin{split}
&\sum_{i=1}^2 \sum_{j=1}^k \rho_{i,j,\varepsilon}\left(\frac{\tilde{M}_{i,j,\varepsilon}}{2\pi}-2\right)\left( \int_{\Omega_j\setminus B_{\theta_{j,\varepsilon}}(x_{j,\varepsilon})} \frac{(e^{f_{i,j,\varepsilon}}-1) }{|x-x_{j,\varepsilon}|^{\frac{\tilde{M}_{i,j,\varepsilon}}{2\pi}}}-\int_{\mathbb{R}^2\setminus \Omega_j} \frac{ 1}{|x-x_{j,\varepsilon}|^{\frac{\tilde{M}_{i,j,\varepsilon}}{2\pi}}}\right)\\
&+\sum_{j=1}^k  B_{j,\varepsilon}e^{\beta_{j,\varepsilon}}+  \sum_{i=1,2}\sum_{j=1,\cdots,k}\frac{(\tilde{M}_{i,j,\varepsilon}-8\pi)^2}{\tilde{M}_{i,j,\varepsilon}-4\pi}\\
=&  O\Big( (\mu_{j,\varepsilon}^{-2}+e^{\beta_{j,\varepsilon}})\mu_{j,\varepsilon}^{-2}\ln \mu_{j,\varepsilon}\theta_{j,\varepsilon}+\theta_{j,\varepsilon}^2\mu_{j,\varepsilon}^{-2} +\mu_{j,\varepsilon}^{-4\tau} +e^{2\tau\beta_{j,\varepsilon}}             \Big),
\end{split}
\end{equation}
where $ \rho_{i,j,\varepsilon}=\mu_{j,\varepsilon}^{2-\frac{\tilde{M}_{i,j,\varepsilon}}{2\pi}    }  e^{I_{i,j,\varepsilon}+u_{0,i}(x_{j,\varepsilon})}$.
By Remark \ref{rem51}, we can rewrite \eqref{e088} as
\begin{equation*}
\left(\mathcal{D}^{(2)}({\bf q}) +o(1)\right) e^{u_{0,1}(q_j)}e^{-\ln \varepsilon^2+ \frac{1}{|\Omega|}\int_{\Omega} u_{1,\varepsilon}}  +B e^{2\ln\varepsilon^2- \frac{1}{|\Omega|}\int_{\Omega} u_{1,\varepsilon}}  +\sum_{i=1,2}\sum_{j=1,\cdots,k}\frac{(\tilde{M}_{i,j,\varepsilon}-8\pi)^2}{\tilde{M}_{i,j,\varepsilon}-4\pi}=0,
\end{equation*}
for some $B>0$. Thus, we find $\mathcal{D}^{(2)}({\bf q}) \leq 0. $
\end{proof}

\medskip

It is easy to see that Theorem \ref{main1} follows immediately from Lemma \ref{final-lem}.

\section*{Acknowledgments}
The authors would like to express their gratitude to  the anonymous referee for valuable comments
and suggestions.

\end{document}